 \documentclass[12pt]{amsart}
\usepackage{epsfig,color}
\usepackage{amssymb}
\usepackage{caption}

\makeatother

\theoremstyle{definition}

\def\fnum{equation} 
\newtheorem{Thm}[\fnum]{Theorem}
\newtheorem{Cor}[\fnum]{Corollary}

\newtheorem{Lem}[\fnum]{Lemma}
\newtheorem{Con}[\fnum]{Conjecture}

\newtheorem{Pro}[\fnum]{Proposition}

\numberwithin{equation}{section}

\newcommand{\Vol}{{\text{Vol}}}

\newcommand{\nn}{{\bf{n}}}

\newcommand{\dist}{{\text {dist}}}

\newcommand{\Hess}{{\text {Hess}}}

\def\RR{{\bold R}}

\def\SS{{\bold S}}

\newcommand{\dv}{{\text {div}}}
\newcommand{\e}{{\text {e}}}

\newcommand{\Hx}{{\Hess^x \! \! }}

\newcommand{\cC}{{\mathcal{C}}}

\newcommand{\cA}{{\mathcal{A}}}

\newcommand{\cL}{{\mathcal{L}}}

\newcommand{\cM}{{\mathcal{M}}}

\newcommand{\cS}{{\mathcal{S}}}

\newcommand{\eqr}[1]{(\ref{#1})}

\title[Arnold-Thom gradient conjecture for the arrival time]{Arnold-Thom gradient conjecture for the arrival time}
\author[]{Tobias Holck Colding}%
\address{MIT, Dept. of Math.\\
77 Massachusetts Avenue, Cambridge, MA 02139-4307.}
\author[]{William P. Minicozzi II}%
 
\thanks{The  authors
were partially supported by NSF Grants DMS 1404540,  DMS 1206827 and DMS 1707270.}

\email{colding@math.mit.edu  and minicozz@math.mit.edu}

\begin{document}

\maketitle

\begin{abstract}

  We prove  conjectures of  Ren\'e Thom and Vladimir Arnold    for  $C^2$ solutions to the degenerate elliptic equation that is the level set equation for motion by mean curvature.  
  
  We believe these results are the first instances of a  general principle:  Solutions of many degenerate equations   behave as if they are analytic, even when they are not.  If so, this would explain various conjectured phenomena.  
\end{abstract}

\section{Introduction}

By a classical result, solutions of analytic elliptic PDEs, like the Laplace equation, are analytic.   Many important equations are degenerate elliptic and solutions have much 
lower regularity.   Still,
 one may hope that solutions share  properties of analytic functions.   On the surface, such properties seem to  be purely analytic; however, they turn  out to be 
  closely connected to important open problems in geometry.

 For an analytic function, Lojasiewicz, \cite{L1}, proved that any gradient flow line with a limit point   has finite length and, thus, limits to a unique critical point.    This result has since been known as {\emph{Lojasiewicz's theorem}}.  
The proof relied on two   {\emph{Lojasiewicz  inequalities}} for analytic functions that had also been used to prove two conjectures around 1960: Laurent Schwarz's division conjecture in 1959 in \cite{L3} and 
a conjecture of Whitney about singularities in 1963 in \cite{L4}. Around the same time, in 1958, 
   H\"ormander proved a special case of Schwarz's division conjecture by establishing Lojasiewicz's first inequality for polynomials, \cite{Ho}.

    \begin{figure*}[htbp]
\centering\includegraphics[totalheight=.18\textheight, width=.50\textwidth]{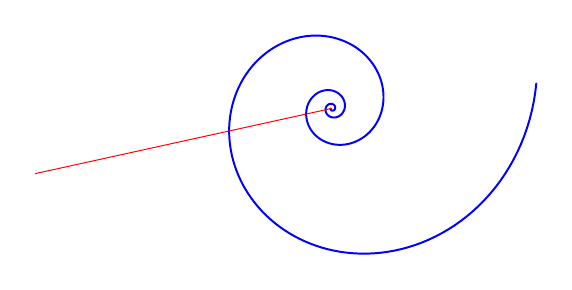}
\caption*{{\bf{Figure}} illustrates in $\RR^3$ a situation   conjectured to be impossible.      The Arnold-Thom conjecture asserts that a blue 
integral curve  does  not spiral  as it approaches the   critical set (illustrated in red,  orthogonal to the plane where the curve spirals). }   
  \end{figure*}

 Around 1972,   Thom, \cite{T}, \cite{L2}, \cite{Ku}, \cite{A}, \cite{G}, conjectured a strengthening of Lojasiewicz's theorem, asserting that  each gradient flow line of an analytic function approaches its limit from a unique limiting direction:
 
 \begin{Con}	\label{c:conj0}
 If a gradient flow line $x(t)$ for an analytic function  has a limit point, then 
  the limit of secants $\lim_{t \to \infty} \, \frac{x(t) - x_{\infty} } { |  x(t) - x_{\infty}|}$ exists. 
 \end{Con}

    This 
 conjecture  arose 
 in  Thom's work on catastrophe theory and singularity theory  and became known as {\emph{Thom's gradient conjecture}}.  The conjecture  was finally proven in 2000 by Kurdyka, Mostowski, and Parusinski in \cite{KMP},
  but the following stronger conjecture remains open  (see page $282$ in Arnold's problem list,   \cite{A}):

\begin{Con}  	\label{c:conj}
  If a gradient flow line $x(t)$ for an analytic function  has a limit point, then the limit
  of the unit tangents $\frac{x'(t)}{|x'(t)|}$ exists. 
 \end{Con}
 
 It is easy to see that if $\lim_{t \to \infty} \,  \frac{x'(t)}{|x'(t)|}$ exists, then so does   $\lim_{t \to \infty} \, \frac{x(t) - x_{\infty} } { |  x(t) - x_{\infty}|}$.  It follows that the {\emph{Arnold-Thom conjecture}}
 \ref{c:conj} implies Thom's gradient conjecture \ref{c:conj0}.   Easy examples show that the
Lojasiewicz  theorem,  the Lojasiewicz inequalities, and both Conjectures \ref{c:conj0} and \ref{c:conj}   fail for general smooth functions; see, e.g.,  fig.~$3.5$ in \cite{Si} or fig.~$1$ in \cite{CM8}.

\vskip2mm
Analytic functions play an important role in  differential equations since solutions of   analytic elliptic equations are themselves analytic.  In many instances, the properties that come from being analytic are more important than analyticity itself.   We will show that solutions of an important degenerate elliptic equation  have analytic properties even though  solutions are not even $C^3$.
Namely, 
  we will show that Conjectures \ref{c:conj0}, \ref{c:conj} hold for solutions  of the classical degenerate elliptic equation, known as the {\emph{arrival time equation}},
 \begin{align}	\label{e:arrivalu}
	-1 = |\nabla u|\,\text{div}\left( \frac{\nabla u}{|\nabla u|}\right)\,  .
\end{align}
Here $u$ is defined on a compact connected subset of $\RR^{n+1}$ with smooth mean convex boundary.  Equation \eqr{e:arrivalu} is the prototype for a family of equations, see, e.g., \cite{OsSe},
 used for tracking moving interfaces in complex situations.  These equations have been instrumental in  applications, including semiconductor processing, fluid mechanics, medical imaging, computer graphics, and material sciences.

Even though   solutions  of \eqr{e:arrivalu} are   a priori only in the viscosity sense,  they are
 always twice differentiable by \cite{CM5}, though not necessarily $C^2$; see \cite{CM6}, \cite{H2}, \cite{I}, \cite{KS}. 
Even  when a solution is $C^2$,   it still might not be $C^3$, Sesum, \cite{S}, let alone analytic as in Lojasiewicz's theorem.  However, solutions behave like analytic functions are expected to:
   
  \begin{Thm}	\label{t:thom}
    The Arnold-Thom conjecture holds for  $C^2$ solutions of \eqr{e:arrivalu}.
  \end{Thm}

The geometric meaning of  \eqr{e:arrivalu} is that the level sets $u^{-1} ( t)$  are mean convex and evolve by mean curvature flow.  
One says that $u$ is the {\emph{arrival time}} since $u(x)$ is  the time the hypersurfaces $u^{-1}(t)$ arrive at $x$ under the mean curvature flow; see
   Chen-Giga-Goto, \cite{ChGG}, Evans-Spruck, \cite{ES}, Osher-Sethian, \cite{OsSe},  and \cite{CM3}.  
Geometrically, 
singular points for the   flow correspond to critical points for $u$.

We conjecture that even for solutions that are not $C^2$, but merely twice differentiable, the Arnold-Thom conjecture holds:

\begin{Con}
Lojasiewicz's inequalities and the Arnold-Thom conjecture hold for all solutions of \eqr{e:arrivalu}.
\end{Con}

If this conjecture holds, then   the gradient Lojasiewicz inequality would imply that the flow is  singular at only finitely many times as has been conjectured,  \cite{W3}, \cite{AAG}, \cite{Wa}, \cite{M}.

\vskip2mm
One of the important ingredients in the proof of Theorem \ref{t:thom} is an essentially sharp rate of convergence for the rescaled mean curvature flow; this will be given in Proposition
\ref{c:djsums1}
below.  This rate is not fast enough to directly show the convergence of unit tangents, which is closely related to the existence of a non-integrable kernel of the linearized operator.  However,
we  overcome this by a careful analysis of this kernel.

\vskip1mm
  We believe that the principle that solutions of degenerate equations behave as though they are analytic, even when they are not, should be quite general.  For instance, there should be versions for other flows, including Ricci flow; cf. \cite{CM9}.

\section{Lojasiewicz theorem for the arrival time}	\label{s:s1}

  A function $v$ satisfies a gradient Lojasiewicz inequality near a point $y$ (see, e.g., \cite{CM8}) if there exist $p> 1$,  $C$ and a neighborhood of $y$ (all depending on $v$ and $y$) so that  
\begin{align}	\label{e:lou}
	|v - v(y)| \leq C \, |\nabla v|^p \, .
\end{align}
This is nontrivial only if $y$ is a critical point.  
If $\nabla v (y) = 0$ and $v$ satisfies \eqr{e:lou}, then 
  $v(y)$ is the only critical value in this neighborhood (this applies for any $p>0$).

  In this section, we show \eqr{e:lou} with $p=2$ for a   $C^2$ solution $u$ of  \eqr{e:arrivalu}.
 When $u$ is not $C^2$, then \eqr{e:lou} can fail for any fixed $p> 1$.  Namely, for any odd integer $m\geq 3$,     Angenent and Vel\'azquez construct rotationally symmetric examples in \cite{AV} where  $|u-u(y)| \approx |\nabla u|^{ \frac{m}{m-1}}$ for a sequence of points tending to $y$.  The examples in \cite{AV} were constructed to analyze so-called type II singularities that were previously observed by Hamilton and proven rigorously to exist by Altschuler-Angenent-Giga,   \cite{AAG}; cf. also \cite{GK}.

 \vskip2mm
 From now on, $u$ will be $C^2$. To prove \eqr{e:lou}, we first   recall   the properties that we will use.  
Namely, if   $\cS = \{ x \, | \, \nabla u(x) = 0 \}$ denotes  the critical set,{\footnote{The flow  is smooth away from the singular set $\cS$  consisting of cylindrical singularities; see, \cite{W1}, \cite{W2},   \cite{H1}, \cite{HS1}, \cite{HS2}, \cite{HaK}, \cite{An}; cf. \cite{B}, \cite{CM1}.  See also \cite{CM4}.}} 
then \cite{CM5} and \cite{CM6} give:
\begin{enumerate}
\item[($\cS 1$)]   $\cS$ is a closed embedded connected $k$-dimensional $C^1$ submanifold whose tangent space is   the kernel of $\Hess_u$.   Moreover, $\cS$ lies in the interior of the region where $u$ is defined.
\item[($\cS 2$)] 
 If $q \in \cS$, then
	$\Hess_u (q) = - \frac{1}{n-k} \, \Pi$ and $\Delta u (q) = - \frac{n+1-k}{n-k}$, 
 where $\Pi$ is orthogonal projection onto the orthogonal complement of the kernel.
 \end{enumerate}
After subtracting a constant, we can assume that   $\sup u = 0$.

Using these properties, the next theorem gives the gradient Lojasiewicz inequality.  

\begin{Thm}	\label{l:gL}
We have that $u(\cS) = 0$ and 
\begin{align}	\label{e:uniqun}
 	\frac{|\nabla u|^2}{-u} \to \frac{2}{n-k} {\text{ as }} u \to 0  \, .
 \end{align}
 In particular, there exists $C>0$ so that $C^{-1} \, |\nabla u|^2 \leq -u \leq C \, |\nabla u|^2$.
 \end{Thm}

 \begin{proof}
The boundary of the domain is smooth and mean convex, so $\nabla u \ne 0$  on the boundary.  The normalization $\sup u =0$ implies that $u=0$ at any maximum.  Thus, there is at least one point in $\cS$ with $u=0$.  By ($\cS 1$),  $u$ is constant on $\cS$ and we conclude that $u(\cS) = 0$.
 
 Given $\epsilon > 0$, choose $\delta > 0$ so that $|p-q| < \delta$ implies that $|u_{ij} (p) - u_{ij} (q)| < \epsilon$ and, moreover, so that the $\delta$-tubular neighborhood of $\cS$ does not intersect the boundary of the domain.  Let $q$ be any point with $\dist (q, \cS) < \delta$ and then let $p$ be a point in the compact set $\cS$ that minimizes the distance to $q$ (note that $p$ might not be unique).  
 Since $\cS$ is $C^1$, the minimizing property implies that the vector $q-p$ is orthogonal to the tangent space to $\cS$.  In particular, ($\cS 2$) implies that
 \begin{align}	\label{e:heukk}
 	\Hess_u(p)(q-p) = - \frac{q-p}{n-k} \, .
 \end{align}
 Given $t \in (0,1]$, the fundamental theorem of calculus gives
 \begin{align}
 	\nabla u (p+t(q-p)) &= \int_0^t \Hess_u(p+s(q-p))(q-p) \, ds \, .
 \end{align}
 Combining this with \eqr{e:heukk} and the continuity of the Hessian gives
  \begin{align}	\label{e:gradub}
 	\left| \nabla u (p+t(q-p)) +   t\,  \frac{q-p}{n-k} \right|  \leq  \epsilon \, t \, |q-p| \, .
 \end{align}
 Using this at $t=1$ gives
  \begin{align}	\label{e:gradub2}
 	\left| \nabla u (q) +      \frac{q-p}{n-k} \right|  \leq  \epsilon  \, |q-p| \, .
 \end{align}
 Using the fundamental theorem of calculus on $u$ this time, \eqr{e:gradub} gives that
 \begin{align}
 	\left| u (q) + \frac{|p-q|^2}{2(n-k)} \right| \leq \int_0^1 \left| \langle \nabla u (p+t(q-p)) +   t\,  \frac{q-p}{n-k} , q-p \rangle \right| \, dt \leq \frac{\epsilon}{2} \, |p-q|^2 \, .
 \end{align}
 Since $\epsilon > 0$ is arbitrary, combining the last two inequalities gives \eqr{e:uniqun}.

 The last claim follows from \eqr{e:uniqun} since $\{ u = 0 \} = \{ |\nabla u| = 0\} = \cS$.
  \end{proof}

\vskip2mm
The next theorem shows that the gradient flow lines of $u$  have finite length (this is the Lojasiewicz theorem for $u$), converge to points in $\cS$, and approach $\cS$ orthogonally.   The first claims follow immediately from the gradient Lojasiewicz inequality of Theorem \ref{l:gL}.
Let $\Pi_{\text{axis}}$ denote orthogonal projection onto the kernel of $\Hess_u$.

\begin{Thm}	\label{l:flowlines}
Each flow line $\gamma$  for $\nabla u$ has finite length and limits to a point in $\cS$.
 Moreover, if we parametrize $\gamma$ by   $s\geq 0$ with $|\gamma_s| =1$ and $\gamma (0) \in \cS$, then
\begin{align}
	u ( \gamma (s)) & \approx   \frac{-s^2}{2(n-k)}   \, ,  \label{e:flow1}  \\
	\left|\nabla u   ( \gamma (s))  \right|^2 & \approx \frac{s^2}{(n-k)^2} \, ,   \label{e:flow2}  \\
		\Pi_{\text{axis}} (\gamma_s) &\to 0 \, .   \label{e:flow3}
\end{align}
In particular, for $s$ small, we have  that $\gamma (s) \subset B_{ 2n \, \sqrt{ -u(\gamma(s))}}  (\gamma(0))$.
\end{Thm}

\begin{proof}
   Each point lies on a flow line where $u$ is increasing and limits to $0$, so $\gamma$ limits to $\cS$.  If we parametrize $\gamma$ by time $t$ (so that $  u\circ \gamma (t) =t $ and $|\gamma_t| = \frac{1}{|\nabla u|}$), then the length is  
   \begin{align}	\label{e:sT}
   	\int_T^0 \frac{1}{|\nabla u|} \, dt \approx \sqrt{ \frac{n-k}{2} } \, \int_T^0 \frac{1}{ \sqrt{-u}} \, dt = \sqrt{ \frac{n-k}{2} } \, \int_T^0 \frac{1}{ \sqrt{-t}} \, dt =  \sqrt{ 2\, (k-n)\,T } \, ,
   \end{align}
   where the approximation used \eqr{e:uniqun}.
   In particular, the flow lines starting from $u= T$ have finite length approximately equal to $ \sqrt{ 2\, (k-n)\,T }$.    It follows that $\gamma$ has a limit $\gamma (0) \in \cS$ as $t \to 0$ and 
   we get the approximation \eqr{e:flow1}.  Combining \eqr{e:flow1} and \eqr{e:uniqun} gives \eqr{e:flow2}.

For $s> 0$, the arrival time equation \eqr{e:arrivalu}, continuity of $\Delta u$, and ($\cS 2$) give  that 
 \begin{align}
 	\Hess_u (\gamma_s , \gamma_s) = \frac{ \Hess_u ( \nabla u , \nabla u) }{|\nabla u|^2} = \Delta u (\gamma(s)) + 1  \to  \Delta u(\gamma(0)) + 1  =  - \frac{1}{n-k} \, .
 \end{align}
Since $\Hess_u \to - \frac{1}{n-k}  \, \Pi$, we conclude that $\Pi_{\text{axis}} (\gamma_s) \to 0$, giving the third claim.  Finally, the last claim follows from \eqr{e:flow1} and $|\gamma_s| = 1$.
\end{proof}

\section{Reducing Theorem \ref{t:thom} to an estimate for rescaled MCF}

   In this section, we will reduce  the Arnold-Thom conjecture to an estimate for rescaled mean curvature flow.

A one-parameter family of hypersurfaces $M_{\tau}$ evolves by {\emph{mean curvature flow}} (or {\emph{MCF}}) if each point $x(\tau)$ evolves by 
$\partial_{\tau} x =   -  H  \, \nn$.  Here $H$ is the mean curvature and $\nn$ a unit normal.  
 The rescaled MCF  
$ 	\Sigma_t = \frac{1}{\sqrt{-u}} \, \{ x   \, | \, u(x) = - \e^{-t} \} 
$ is equivalent to simultaneously running MCF and rescaling space, up to reparameterizations of time and the hypersurfaces.
A one-parameter family of hypersurfaces $\Sigma_t$
flows 
by the {\emph{rescaled  MCF}}  if  
\begin{align}
	\partial_t x =   - \left( H - \frac{1}{2} \, \langle x , \nn \rangle \right) \, \nn   \, .
\end{align}
It will be convenient to set $\phi = H - \frac{1}{2} \, \langle x , \nn \rangle$.
The fixed points for rescaled MCF are shrinkers where $\phi =0$;  the most important examples are cylinders
$\cC = \SS^{n-k}_{\sqrt{2(n-k)}} \times \RR^k$ where $k = 0 , \dots , n-1$.   Below, $\Pi:\RR^{n+1} \to  \RR^{n-k+1}$ is orthogonal projection on 
the orthogonal complement of the axis $  \RR^k$ of the cylinder $\cC$.
The rescaled MCF   is the negative gradient flow for the Gaussian area  
\begin{align}
	F (\Sigma) \equiv \int_{\Sigma} \e^{ - \frac{|x|^2}{4} }   \, .
\end{align} 
In particular, $F(\Sigma_t)$ is non-increasing.
Define the sequence $\delta_j$ by
\begin{align}	\label{e:deltaj}
	\delta_j = \sqrt{F (\Sigma_{j-1}) - F(\Sigma_{j+2})} \, .
\end{align}
As in \cite{CM1}, the entropy $\lambda (\Sigma)$ is   $  \sup_{t_0 > 0 , x_0 \in \RR^{n+1}} \, F( t_0 \Sigma + x_0)$. We will use
the Gaussian $L^p$ norm  given by
 $
  	\| g \|_{L^p (\Sigma_t)}^p \equiv  \int_{\Sigma_t} |g|^p \, \e^{ - \frac{|x|^2}{4} } $.

\subsection{Summability of   $\delta_j$}

As we will see in \eqr{e:phijj1} below, $\delta_j$   bounds the distance that $\Sigma_t$ evolves from $j$ to $j+1$.  Existence of $\lim_{t \to \infty} \Sigma_t$ is proven in \cite{CM2} by showing that $\sum \delta_j < \infty$.
We will need that $\delta_j$ is summable even after being raised to a power  less than one:

\begin{Pro}	\label{c:djsums1}
There exists $\bar{\beta} < 1$ so that  
\begin{align}	\label{e:djsums1}
	\sum_{j=1}^{\infty} \, \delta_j^{ \bar{\beta}} < \infty \, .
\end{align}
\end{Pro}

\begin{proof} 
By (6.21) and lemma $6.9$ in \cite{CM2}, there exists $\rho > 1$ and $C$ so that
\begin{align}
	\sum_{k=j}^{\infty} \delta_k^2 \leq 3 \left( F(\Sigma_{j-1}) - \lim_{t\to \infty} F(\Sigma_t) \right) \leq C \, j^{ - \rho }  \, .
\end{align}
Moreover, lemma $6.9$ in \cite{CM2} shows that this implies that $\sum \delta_j < \infty$.  

We will show next that if $0 < q < \rho $, then
\begin{align}	\label{e:claimq}
	\sum \delta_j^2 \, j^q < \infty \, .
\end{align}
 To prove this,  set $b_j = j^q$ and $a_j = \sum_{i=j}^{\infty} \delta_i^2$, then $a_j - a_{j+1} = \delta_j^2$ and 
\begin{align}
	  b_{j+1} - b_j =  (j+1)^q - j^q \leq c \, j^{q-1} \, ,
\end{align}
where $c$ depends on $q$ and we  used that $j\geq 1$.  Summation by parts  gives
\begin{align}
	\sum_{j=k}^N \delta_j^2 j^q &= \sum_{j=k}^N b_j (a_j - a_{j+1}) = b_k a_k - b_{N} a_{N+1} + 
	\sum_{j=k}^{N-1} a_{j+1} ( b_{j+1}- b_j ) \notag \\
	&\leq k^q \, \sum_{j=k}^{\infty} \delta_j^2  + C\, \sum_{j=k}^{\infty} j^{- \rho} \, j^{q-1} \, .
\end{align}
This is bounded independently of $N$ since $q < \rho $, giving \eqr{e:claimq}.

Suppose that $a > 0$.  The H\"older inequality gives    $\sum \delta_j^{\beta} = \sum \left( \delta_j^{\beta} \, j^a \right) \, j^{-a} < \infty$ if
\begin{align}	\label{e:twosums}
	\sum \delta_j^2 \, j^{ \frac{ 2a}{\beta} }  + \sum j^{ - \frac{2a}{2-\beta}} < \infty \, .
\end{align}
To get \eqr{e:djsums1}, we need $\beta < 1$ and $a$ so that both sums in \eqr{e:twosums} are finite.  By \eqr{e:claimq}, the first is finite if $\frac{2a}{\beta}< \rho$.  The second is finite if
$  2 -  \beta < 2a$.  To satisfy both, we must have
\begin{align}
	2 -  \beta  < 2a < \rho \,  \beta  \, .
\end{align}
This is possible as long as $2<(1+ \rho) \, \beta$.  Since $1< \rho$, we can choose such a $\beta < 1$.
 \end{proof}

\subsection{Cylindrical approximation}

The rescaled MCF $\Sigma_t$ converges to a limiting cylinder $\cC$ by \cite{CM2}.  Thus, for each large  integer $j$, $\Sigma_j$ is well-approximated by $\cC$. 

  In  the next proposition, we will bound the distance from $\Sigma_t$ to some cylinder $\cC_t$  that is allowed to change with $t$.  We will let $\Pi_t$ denote the projection orthogonal to axis of $\cC_t$. The operator $\cL$ will be the drift Laplacian on the cylinder $\cC_t$.   Property (1) collects a priori estimates for the graph function $w$, (2) shows that $w$ almost satisfies the linearized equation, (3) 
  shows the approximating cylinders converge, and (4) gives   a priori bounds on higher derivatives. We will only use (3) in this section; (1), (2) and (4) will be used later.

\begin{Pro}	\label{p:evolving}
Given $0 < \epsilon_1$ and $ \beta < 1$, there exist  a constant $C $  and  a sequence of radii $R_j$ and cylinders $\cC_j$ satisfying:
\begin{enumerate}
\item For $t \in [j,j+1]$, $\Sigma_t$ is a graph over $B_{R_j} \cap \cC_{j+1}$ of  a function $w$ with  $  \| w \|_{C^4} \leq \frac{\epsilon_1}{R_j}$, 
\begin{align}
	 \| w \|_{ W^{3,2}}^2 + \| \phi \|_{W^{3,2}(B_{R_j})}+ \e^{ - \frac{R_j^2}{4} } &\leq C  \, \delta_j^{\beta} \, , \notag \\
	 w^2 + |\nabla w|^2 + |\Hess_w|^2 + |\nabla \Hess_w|^2 + |\nabla^2 \Hess_w|^2 &\leq C  \, \delta_j^{\beta} \,  \e^{ \frac{|x|^2}{4}} \, , \notag \\
	 \phi^2 + |\nabla \phi|^2 + |\Hess_{\phi}|^2 &\leq C  \, \delta_j^{2\beta} \, \e^{ \frac{|x|^2}{4}} \, . \notag
\end{align}
\item  
The function $w$ and its Euclidean partial derivatives $w_i$ and $w_{ij}$ on $\cC_t$   satisfy
\begin{align*}	 
	\left| \phi - \left( \cL + 1 \right) w \right| &\leq C  (1+R_j)(w^2 + |\nabla w|^2) + C  (|w| + |\nabla w|) \, \left| \Hess_w \right| \, , \\
		\left| \phi_i - \left( \cL + \frac{1}{2} \right) w_i \right| &\leq    C(1+R_j) \,  (|w| + |\nabla w| + \left| \Hess_w \right| +  |\nabla \Hess_w|) \, \left( |w| + |\nabla w|+ |\nabla w_i|   \right)  \, , \\
	\left| \phi_{ij}   - \cL \, w_{ij} \right| & \leq C (1+R_j)\, \left( |w| + |\nabla w| + |\Hess_w| + |\nabla \Hess_w|  + |\nabla^2 \Hess_w| \right)^2 \, .
\end{align*}
\item $\left| \Pi_j - \Pi_{j+1} \right| \leq C  \,\delta_j^{\beta}$. 
\item Given any $\ell$, there exists $C_{\ell}$ with $|\nabla^{\ell} w| + |\nabla^{\ell} \phi |  \leq C_{\ell}$.
\end{enumerate}

\end{Pro}

   \begin{proof} 
   Let $\epsilon_0 > 0$ and $\alpha$ be fixed as in the definition of $r_{\ell}$ on page $261$ in \cite{CM2}.  We will initially find a radius $R_j'$ so that every estimate (1), (2), (3) and (4) holds except for the $C^4$ bound in (1)
   which we replace by $\| w \|_{C^{2,\alpha}} \leq \epsilon_0$.  We will then use (1) to get the $C^4$ bound on a slightly smaller $R_j < R_j'$ with the other  bounds still holding.

    As in $(5.2)$ in \cite{CM2}, define $ R_j'$ by
        $	\e^{ - \frac{(R_j')^2}{2} } = \delta_j^2$.
          Since $\Sigma_t \to \cC$, we can assume that $\Sigma_t$ is fixed close to $\cC$ on a large set.   
Theorem $5.3$ in \cite{CM2} gives  $C$ and $\mu > 0$ and a cylinder $\cC_{j+1}$ so that $B_{(1+2\mu) \, R_j'  - C} \cap \Sigma_t$, for $t \in [j,j+1]$, is a graph over $\cC_{j+1}$ of a function $w$ with $\| w \|_{C^{2,\alpha}} \leq \epsilon_0$ and, moreover, (4) 
         holds.           Furthermore, lemma $5.32$ in \cite{CM2} gives  $C$ so that 
         \begin{align}	\label{e:from532}
         		\int_{B_{(1+\mu)R_j'} \cap \Sigma_t} |\phi|^2 \, \e^{ -\frac{|x|^2}{4}} \leq C \, \delta_j^2 \, .
         \end{align}
         Using theorem $0.24$ from \cite{CM2}, we get for any $\beta_1 < 1$ that
  \begin{align}	\label{e:theold3}
         		\| w \|_{L^2}^2 \leq  C_{\beta_0 , \beta_1} \, \delta_j^{\beta_1} \, .
        \end{align}
        Using the higher derivative bound from (4) and interpolation (e.g., lemma $B.1$ in \cite{CM2}), we get for any $\beta_2 < \beta_1$ that
 \begin{align}	\label{e:theold3w}
         		\| w \|_{W^{3,2}}^2 \leq  C_{\beta_0 , \beta_1,\beta_2} \, \delta_j^{\beta_2}  \, .
        \end{align}
        We have now established the first part of  (1).  Similarly, the second two parts of (1) follow from the first part, (4) and interpolation again.
          
      We turn next to property (2).                 Lemma $4.6$ in \cite{CM2} computes the nonlinear graph equation for shrinkers; using $p$ for points in $\cC_t$, this gives
\begin{align}	\label{e:taylorE}
	\phi = \hat{f} (w,\nabla w) + \langle p , V(w, \nabla w) \rangle + \langle \Phi (w , \nabla w) , \Hess_w \rangle \, , 
\end{align}
where $\hat{f}(s,y)$, $V(s,y)$ and $\Phi (s,y)$ are smooth functions for $|s|$ small.  Moreover, since
  $|A|^2 = \frac{1}{2}$ on $\cC_t$, the operator $\cL +1$ is the linearized operator for the shrinker equation and lemma $4.10$ in \cite{CM2} gives that
\begin{align}	\label{e:philinear}
	\left| \phi - \left( \cL + 1 \right) w \right| \leq C_1 (w^2 + |\nabla w|^2) + C_2 (|w| + |\nabla w|) \, \left| \Hess_w \right| \, , 
\end{align}
where $C_1 \leq C(1+|p|)$ and $C_2$ is bounded.    This gives the first claim in (2).  Differentiating \eqr{e:taylorE} in a Euclidean direction $x_i$ and arguing similarly gives
the second claim.  Finally, differentiating  \eqr{e:taylorE}  again gives the remaining claim in (2).

We next prove (3) by bounding the Gaussian distance from $\Sigma_j$ to $\Sigma_{j+1}$ by $C\, \delta_j$ and showing that   $\cC_j$ is Lipschitz in $\Sigma_j$.  
The first part follows since $|x_t | = |\phi|$ and 
\begin{align}	\label{e:phijj1}
	\int_j^{j+1} \| \phi \|_{L^1} \, dt  \leq C \int_j^{j+1} \| \phi \|_{L^2} \, dt \leq C \left( \int_j^{j+1} \| \phi \|_{L^2}^2 \, dt \right)^{ \frac{1}{2} } \leq C \, \delta_j \, .
\end{align}
To see that $\cC_j$ is Lipschitz in $\Sigma_j$, we need to slightly modify the proof of theorem $0.24$ in \cite{CM2}.  The choice of the cylinder in \cite{CM2} occurs on page $240$  in step $1$ of the proof of proposition $2.1$ there.  There, the $\RR^k$ factor is determined to be the approximate kernel of $A$ at any point $p$ in a fixed ball $B_{2\sqrt{2n}}$. In \cite{CM2}, $p$ is left arbitrary -- it does not effect the bounds in (1), (2) and (4) -- and the $\RR^k$ factor given by choosing any $p$ would work (all that is needed are (2.22)--(2.24) there).
To make $\cC_j$ Lipschitz in $\Sigma_{j}$, we will choose the $\RR^k$ factor by averaging over the approximate kernel of $A$ for each point in the ball $B_{2\sqrt{2n}}$.  The resulting $\RR^k$ factor, and thus the cylinder, is then Lipschitz in $\Sigma_j$ as desired.

Finally, we  will fix $R_j \leq R_j'$ where $\| w \|_{C^4} \leq \frac{\epsilon_1}{R_j}$ and we still have $\e^{ - \frac{R_j^2}{4} } \leq C \, \delta_j^{\beta}$ where $C$ now also depends on $\epsilon_1$.
This follows  from the pointwise $C^4$ bounds in (1) and  $	\e^{ - \frac{(R_j')^2}{4} } = \delta_j$.
  \end{proof}

\subsection{Reduction}
 The next theorem reduces  Theorem \ref{t:thom}  to an estimate for  rescaled MCF.

 \begin{Thm}	\label{t:rMCFA}
 Theorem \ref{t:thom}  holds if every rescaled MCF
  $\Sigma_t$  with  $\lambda (\Sigma_t) < \infty$ that goes to a cylinder   as $t \to \infty$ satisfies
\begin{align}	\label{e:rmcfa}
	\sum_{j=1}^{\infty} \,   \int_j^{j+1} \left( \sup_{B_{2n} \cap \Sigma_t} \left| \Pi_{j+1} (\nabla H) \right| \right) \, dt  < \infty \, .
\end{align}
 \end{Thm}
 
 \vskip1mm
 We will prove  Theorem \ref{t:rMCFA} here and  \eqr{e:rmcfa}  in Section \ref{s:s4}.
  Suppose, therefore, that the function $u$ and reparameterized gradient flow line $\gamma(s)$ are as in Section \ref{s:s1}.
In particular,  $\gamma (s)$ is defined  on  $[0,\ell]$ with $|\gamma_s| = 1$ and $\gamma (0)   \in \cS$.   We will show that $\gamma_s$ has a limit as $s \to 0$.
 The derivative of $\gamma_s = - \frac{\nabla u}{|\nabla u|}$ is  
\begin{align}	\label{e:D3a}
 	\gamma_{ss} &=  - \frac{1}{|\nabla u|} \, \left(  \Hess_u ( \gamma_s) -  \gamma_s \, \langle  \Hess_u ( \gamma_s) , \gamma_s  \rangle	\right)= - \frac{\left( \Hess_u (\gamma_s) \right)^T }{|\nabla u|}=
	\nabla^T \log |\nabla u|   \, ,
 \end{align}
 where $\left( \cdot \right)^T$ is the tangential projection onto the level set of $u$.  
 
 \vskip2mm
 The simplest way to prove that $\lim \gamma_s$ exists would be to show that $\int |\gamma_{ss}| < \infty$, which is related to the rate of convergence for an associated
  rescaled MCF.  While this rate   fails to give integrability of $|\gamma_{ss}|$, it does give the following:
  
  \begin{Lem}	\label{l:seq}
  Given any $\Lambda > 1$, we have
$
  	\lim_{s \to 0} \,  \int_{s}^{\Lambda \, s} |\gamma_{ss} | \, ds  = 0$.
	  \end{Lem}
  
  \begin{proof}
 Using Theorem \ref{l:flowlines} and the fact that $\Hess_u \to - \frac{1}{n-k} \, \Pi$, 
  \eqr{e:D3a} implies that $  s\, |\gamma_{ss} |   \to 0$.  The lemma follows immediately from this.
  \end{proof}
   
  To get around the lack of integrability, we will decompose $\gamma_s$ into two pieces - the parts tangent and orthogonal to the axis - and deal with these separately.  The tangent part goes to zero by
  \eqr{e:flow3} in Theorem \ref{l:flowlines}.  We will use  \eqr{e:rmcfa} to  control  the orthogonal part.

   \begin{proof}[Proof of Theorem \ref{t:rMCFA}]        
Translate so that $\gamma (0) = 0$ and
let $\bar{H}= \frac{1}{|\nabla u|}$ be the mean curvature of the level set of $u$.
The mean curvature $H$  of   $\Sigma_t$ at time $t = - \log (-u)$ is given by
 \begin{align}		\label{e:barH}
 	\bar{\nabla}  \log \bar{H}  = \frac{  {\nabla} \log  {H}  }{ \sqrt{ - u  }}   \approx \frac{ \sqrt{2(n-k)}}{s} \, \ {\nabla} \log {H}   \, .
 \end{align}
    Note that  $u(\gamma (s))$ is decreasing and Theorem \ref{l:flowlines} gives   $t(s) \approx - 2 \, \log s + \log (2(n-k))$ and 
\begin{align}	  	\label{e:tofs2}
 	t'(s) &=- \partial_s \, \left( \log (-u(\gamma(s)) \right) = \frac{ -\partial_s u(\gamma(s))}{ u(\gamma(s)) } \approx - \frac{2}{s} \, .
 \end{align}
 Given a positive integer $j$, define $s_j$ so that $t(s_j) = j$.  Note that  $\left| \log \frac{s_{j+1}}{s_j} \right|$ is uniformly bounded.  
 Therefore, by Lemma \ref{l:seq}, it suffices to show that $\gamma_{s_j}$ has a limit.  
 
 We can write $\gamma_{s_j} =\Pi_{\text{axis},j} (\gamma_{s_j}) + \Pi_j (\gamma_{s_j})$.  We have $\Pi_{\text{axis},j} (\gamma_{s_j})  \to 0$ since
 $\Pi_{\text{axis},j} \to \Pi_{\text{axis}}$ and $\Pi_{\text{axis}} (\gamma_s) \to 0$.  Thus,  we need that
 $\lim_{j \to \infty} \Pi_j (\gamma_{s_j}) $ exists; this will follow   from
 \begin{align}	\label{e:toprovej}
 	\sum_j \, \left| \Pi_j (\gamma_{s_j}) -  \Pi_{j+1} (\gamma_{s_{j+1}}) \right| < \infty \, .
 \end{align}

 Theorem \ref{l:flowlines} gives (for $s$ small) that $\gamma (s) \subset B_{ 2n \, \sqrt{ -u(\gamma(s))}}  $ and, thus,
  \eqr{e:D3a} gives
\begin{align}	 
	\left| \Pi_{j+1} (\gamma_{s_j}) -  \Pi_{j+1} (\gamma_{s_{j+1}}) \right|  &\leq \int_{s_{j+1}}^{s_j} \left| \Pi_{j+1} (\gamma_{ss} ) \right| \, ds = \int_{s_{j+1}}^{s_j}  \left| \Pi_{j+1}\left( \bar{\nabla}  \log \bar{H} (\gamma(s)) \right) \right| \, ds  \notag \\
	& \leq C \, \int_{s_{j+1}}^{s_j}  \sup_{B_{ 2n \, \sqrt{-u(\gamma(s))}}}\, \left| \Pi_{j+1} (\nabla \log  \bar{H}) \right| ( \cdot ,-u) \, ds  \, .	\label{e:aaaa}
\end{align} 
Using \eqr{e:barH} and \eqr{e:tofs2} in \eqr{e:aaaa} and then applying Theorem \ref{t:rMCFA} gives
\begin{align}	 
	\sum_j \, \left| \Pi_{j+1} (\gamma_{s_j}) -  \Pi_{j+1} (\gamma_{s_{j+1}}) \right| &   
	 \leq C \, \sum_j \,  \int_j^{j+1} \sup_{B_{ 2n} \cap \Sigma_t } \, \left| \Pi_{j+1} (\nabla   {H}) \right|  \, dt < \infty \, .
\end{align} 
  On the other hand, $\sum_j \left| \Pi_{j} (\gamma_{s_j}) -  \Pi_{j+1} (\gamma_{s_{j}}) \right|   < \infty$ by (3) in Proposition \ref{p:evolving} and Proposition \ref{c:djsums1}.
 	  Therefore, the triangle inequality gives \eqr{e:toprovej}, completing
  the proof.
               \end{proof}

\section{Approximate eigenfunctions on cylinders}

The key remaining point is summability of $\Pi_{j+1} (\nabla H)$.
  The bound for $w^2$ in (1) from Proposition \ref{p:evolving} is  summable by Proposition \ref{c:djsums1}, but the bound for $w$   is not.     In particular, (1) gives a bound for $\nabla H$ that is not summable.   
This bound for $\nabla H$  cannot be improved  due to slowly growing   Jacobi fields.  However,  these Jacobi fields do not contribute to  $\Pi_{j+1}(\nabla H)$.  We will show that the remainder of $w$, after we subtract these Jacobi fields, is small.

In this section, we will show that if an approximate eigenfunction $w$ on  a cylinder $\cC$  begins to grow, then it must grow  rapidly.  The key tool is the frequency function for the drift Laplacian
as in \cite{CM7}; the  difficulty here is handling  error terms.
 Let $x \in \RR^{k}$ be coordinates on the Euclidean factor, $f = \frac{|x|^2}{4}$,     $\cL$  the drift Laplacian 
 $
  	\cL = \Delta_{\cC} - \frac{1}{2} \, \nabla_x = \Delta_{\theta} + \cL_{\RR^k}  
$, where $\Delta_{\theta}$ is the Laplacian on  $\SS^{n-k}_{\sqrt{2(n-k)}} $, and $\dv_f = \dv - \langle \frac{x}{2} , \cdot \rangle$ the drift divergence.

In applications,  $w$ will be given by Proposition \ref{p:evolving} and, thus, will satisfy (1), (2) and (4) there.  Thus, 
we will assume that  $w$ is a function on $\{ |x| < R \} \subset \cC$ satisfying:
\begin{align}	\label{e:oldONE}
	\left|   ( \cL +1)w - \phi \right| &\leq  \epsilon\, \left( |w| + |\nabla w| \right)   {\text{ where $\phi$ is a function and  $\frac{8}{9} < (1-3\epsilon)^3$}} \, , \\
	\sup_{|x| < 4n} \left|   ( \cL +1)w   \right| & \leq \mu \, . \label{e:new3p1}
\end{align}
Equation \eqr{e:oldONE} arises from  $w$ satisfying a  nonlinear equation $\cM w = \phi$ and  $\cL + 1$ is the linearization of $\cM$. 
We will also assume that $\mu > 0$ is small and
 \begin{align}	\label{e:w11}
	\| w \|_{W^{3,2}}^2 + \| \phi \|_{W^{3,2}} + \e^{ - \frac{R^2}{4} } &\leq \mu \, , \\
	w^2 (x) + |\nabla w(x)|^2 + \left| \Hess_w (x) \right|^2 &\leq \mu \, \e^{ \frac{|x|^2}{4}} \, .	\label{e:w22}
\end{align}
We will assume that the Euclidean first derivatives $w_i$ and second derivatives $w_{ij}$ satisfy
\begin{align}	
	\left| \left( \cL + \frac{1}{2} \right)    \, w_{i} \right| &\leq |\phi_i| + \epsilon \, \left( |w| + |\nabla w| + |\nabla w_i| \right)   \, , \label{e:thenewguy} \\
	\sup_{|x| < r} \, \left| \cL   \, w_{ij} \right| &\leq C_r \,  \mu  {\text{ where $C_r$ depends on $r$}}  \, . \label{e:wijbounds}
\end{align}

By lemma $3.26$ in \cite{CM2}, the kernel of  $\cL +1$ on the weighted Gaussian space on $\cC$ consists of quadratic polynomials and ``infinitesimal rotations'' of the form
\begin{align}	\label{e:kernelcC}
	\tilde{w} = \sum_i a_i (x_i^2 - 2) + \sum_{i<j} a_{ij} x_i x_j + \sum_k   x_k h_k (\theta) \, , 
\end{align}
where $a_i , a_{ij}  $ are constants and each $h_k $ is a $\Delta_{\theta}$-eigenfunction with eigenvalue $\frac{1}{2}$.

The next theorem  quadratically  approximates $w$ in $|x| \leq 3n$ by   $\tilde{w}$ as in \eqr{e:kernelcC}.  Namely, while \eqr{e:w22} gives $|w  | \leq C \, \mu^{ \frac{1}{2} }$,   the next theorem  gives $|w - \tilde{w}| \leq C \, \mu^{ \nu }$ with $\nu \approx 1$.

\begin{Thm}	\label{t:goalCM7}
Given $\nu < 1$, there exists $C$, $\bar{\epsilon}$, and $\mu_0 > 0$ so that if $w$ satisfies \eqr{e:oldONE}--\eqr{e:wijbounds} with  $\mu_0  >  \mu$ and $\bar{\epsilon} > \epsilon$, then there is a function $\tilde{w}$ as in \eqr{e:kernelcC} with
\begin{align}	\label{e:tgoal}
	\sup_{|x| \leq 3n} \, \left| w - \tilde{w} \right| \leq C \,\mu^{ \nu }  \, .
\end{align}
\end{Thm}

\subsection{First reduction}

\begin{Lem}		\label{l:taylorsphere}
If $w$ satisfies \eqr{e:oldONE}--\eqr{e:wijbounds}, then  there is a function $\tilde{w}$ as in \eqr{e:kernelcC} so that $v = w - \tilde{w}$
 satisfies 
\eqr{e:oldONE}--\eqr{e:wijbounds} and 
\begin{itemize}
 \item[(A1)] Each Euclidean second derivative $v_{ij}$ has $\int_{x=0} v_{ij} = 0$.
 \item[(A2)] Each Euclidean first derivative $v_i$ has $\int_{x=0} v_i \,h = 0$ for any $h$ with $\Delta_{\theta} h = - \frac{1}{2} \, h$.
 \item[(A3)] We have $\left| \int_{x=0} v \right| \leq  \mu \, \Vol ( x=0)$.
\end{itemize}
\end{Lem}

\begin{proof}
 Given $\tilde{w}$ as in \eqr{e:kernelcC}, the Euclidean first and second derivatives are given at $x=0$ by 
 \begin{align}
 	\tilde{w}_i = h_i (\theta) , \, \tilde{w}_{ii} =2 a_i , \, \tilde{w}_{ij} = a_{ij} {\text{ for }} i < j \, .
 \end{align}
To arrange (A1), define $a_i$ and $a_{ij}$ by  
 \begin{align}
 	2a_i  \, \int_{x=0} 1 = \int_{x=0} w_{ii} {\text{ and }} a_{ij} \, \int_{x=0} 1 = \int_{x=0} w_{ij} {\text{ for }} i < j \, .
 \end{align}
 Similarly, for (A2),  let $h_i$ be the projection of $w_i$ onto the $\frac{1}{2}$-eigenspace of $\Delta_{\theta}$ at $x=0$.  Claim (A3) follows by integrating  
  $(\cL + 1)w$ at $x=0$ and using \eqr{e:new3p1} and (A1).
\end{proof}

 For a function $v$, we let $\Hx_v = \frac{ \partial^2 v}{\partial x_i \partial x_j}$ denote its Euclidean Hessian.

\begin{Cor}	\label{c:bartime}
Given $\beta > 2$, there exists  $C$ so that
if $v$ satisfies \eqr{e:oldONE}--\eqr{e:wijbounds} and (A1)--(A3), then   
\begin{align}	\label{e:CbetarAA}
	\sup_{|x| < 3n} \, \left(  |v|^{\beta} + \left|\nabla_{\RR^k} v \right|^{\beta} + \left| \Hx_ v\right|^{\beta} \right) \leq C \, \mu^2 + C \,  \int_{ |x| < 4n } \left| \Hx_v \right|^2 \, .
\end{align}
Similarly, given $\beta > 2$ and $ r > 3n$, there exists $C_{\beta , r}$ so that
\begin{align}	\label{e:Cbetar}
	\sup_{|x| < r} \, \left(  |v|^{\beta} + \left|\nabla_{\RR^k} v \right|^{\beta} + \left| \Hx_ v\right|^{\beta}  \right) \leq C_{\beta , r} \, \mu^2 + C_{\beta , r} \,  \int_{ |x| < r+1 } \left| \Hx_v \right|^2 \, .
\end{align}
\end{Cor}

\begin{proof}
We will prove \eqr{e:CbetarAA}; \eqr{e:Cbetar} follows similarly.
Set $\delta^2 = \mu^2 +    \int_{ |x| < 4n } \left| \Hx_v \right|^2$.
Since we have uniform higher derivative bounds on $v$, interpolation implies that all norms are equivalent if we  go to any worse power.  Thus, given any $\beta_1 < 1$, 
\eqr{e:new3p1} gives
\begin{align}	\label{e:e413}
	\left\| \Hx_v \right\|_{C^2}  + \left\| (\cL + 1) v \right\|_{C^2}  \leq C_1 \, \delta^{\beta_1} \, , 
\end{align}
where $C_1$ depends on $\beta_1$.  It follows that  
$
	\left| (\Delta_{\theta} +1 ) v (\theta , 0) \right| \leq C_1 \,  \delta^{\beta_1} $.
	Since $\Delta_{\theta} +1$ is invertible (lemma $2.5$ in \cite{CM2}), this (and interpolation again) gives for any $\beta_2 < \beta_1$ that
\begin{align}	\label{e:vtheta0}
	\left| v (\theta , 0) \right| \leq C_2 \,  \delta^{\beta_2} \, .
\end{align}
Given a Euclidean first derivative $v_i$, 
  \eqr{e:e413} gives that
  \begin{align}
	\left| \left( \Delta_{\theta} + \frac{1}{2} \right) v_i (\theta , 0) \right| \leq C_1 \,  \delta^{\beta_1} \, .
\end{align}
The operator $\left( \Delta_{\theta} + \frac{1}{2} \right)$ is not invertible, but 
 (A2) implies that $v_i (\theta , 0)$ is orthogonal to the kernel so we get   (using interpolation again)
that
	$\left| v_i (\theta , 0) \right| \leq C_2 \,  \delta^{\beta_2}$.
	The bound on $v_i$ at $x=0$ and the Hessian bound give a bound on $v_i$ everywhere.  Integrating this and using \eqr{e:vtheta0} gives the desired pointwise bound on $v$,
completing the proof.
\end{proof}

\subsection{The frequency}

  Given a function $u$ on $\cC$,   define $I$ and $D$ by{\footnote{When $k=1$ and the sphere is disconnected, let $r$ be signed distance and set $I(|r|) = \int_{x=r} u^2 + \int_{x=-r} u^2$.}}
  \begin{align}
  	I(r) &= r^{1-k} \int_{|x| = r} u^2 \, , \\
	D(r) &= r^{2-k} \, \int_{|x| = r} u \, u_r = \e^{ \frac{r^2}{4}} \, r^{2-k} \, \int_{|x| < r } \left(|\nabla u|^2 + u \cL \, u \right) \, \e^{-f} \, .
  \end{align}
  Here $u_r$ denotes the normal derivative of $u$ on the level set $|x|  = r$.  Note that $f$ is proper.
    It is easy to see that $I' = \frac{2D}{r}$ and $(\log I)' = \frac{2U}{r}$, where 
  the frequency  $U = \frac{D}{I}$; cf. \cite{Be}, \cite{CM7}.
   
  \vskip1mm
The next theorem shows that if the growth of an  approximate eigenfunction    hits a certain threshhold, then it  grows very rapidly.  The theorem is stated for eigenvalue $1$, but generalizes easily to other eigenvalues.   The case $(\cL +1)v=0$, where $\epsilon = \phi = 0$, follows from \cite{CM7}.

\begin{Thm}	\label{t:app}
Given $r_1 > \max \{ 9n , 4n+ 64\sqrt{2} \}$, 
there exist $\bar{R} = \bar{R}(n,r_1)$, $C = C(n,r_1)$  so that if
 $v$  is a function on $\{ |x| \leq R \}$ satisfying \eqr{e:oldONE}, where $\bar{R} \leq R   $,     and 
 \begin{align}	\label{e:holefilledA}
 	2 \, \int_{|x| < 9n } v^2 \, \e^{-f} \leq \int_{|x| < r_1} v^2 \, \e^{-f}  {\text{ and }}  \frac{r_1^2}{16} < U(r_1) \, ,
 \end{align}
then  for any $\Lambda \in (0, 1/3)$  
\begin{align}	\label{e:goalA}
	\int_{|x| < {4n}} v^2 \, \e^{-f}  \leq   \Lambda^{-2} \, \| \phi \|_{L^2}^2 + C \, I(R) \, R^{2n+68} \, \e^{ - \frac{ \left(1-3\epsilon-  \Lambda  \right)\, R^2}{2\left(1+3\epsilon+ \Lambda  \right)^2}}  \, .
\end{align}
\end{Thm}

\vskip2mm
To prove Theorem \ref{t:goalCM7}, we will find a scale $r_1$ where Theorem \ref{t:app} applies to give that $w$ is bounded by $\mu^{\nu}$.  To do this, we will find a long stretch where  $\Hx_w$ must grow and, thus,  $w$  must also have grown.  Note  that $\Hx_w$ is easier to work with since each $\RR^k$ derivative lowers the eigenvalue by $1/2$ and, thus, lowers the threshold for growth (cf. \cite{CM7}).

\vskip2mm
The proof of Theorem \ref{t:app} uses a modified version of the frequency.  Define  $E $ and  $U_E$ by 
\begin{align}
	E(r)  &= 
	 r^{2-k} \, \e^{ \frac{r^2}{4} } \, \int_{|x| <r} \left\{ |\nabla v|^2  -v^2   \right\} \e^{-f}= D(r) - r^{2-k} \, \e^{ \frac{r^2}{4} } \, \int_{|x|<r}   \left( v\cL v  + v^2   \right)    \e^{-f}  \, ,  \\
	U_E(r) &= \frac{E(r)}{I(r)} \, .
\end{align}

\begin{Lem}	\label{l:diffineq}
If $E(r) > 0$, then 
\begin{align}
	\left( \log U_E \right)'  (r)   \geq 
	\frac{2-k}{r}   + \frac{r}{2}  - \frac{r}{U_E} +  \frac{U(r)}{r} \left(  \frac{D(r)}{E(r)} - 2  \right) \, .
\end{align}
\end{Lem}

\begin{proof}
The  Cauchy-Schwarz inequality  $\left( \int u u_r \right)^2 \leq \int u^2 \,\int |\nabla u|^2$ gives  
\begin{align}
	E'(r) = \frac{2-k}{r} \, E + \frac{r}{2} \, E   +  r^{2-k} \, \int_{|x|=r} (  |\nabla v|^2 -v^2) \geq 
	\frac{2-k}{r} \, E + \frac{r}{2} \, E  +  \frac{UD}{r} -  r\, I \, .
\end{align}
  The lemma follows from this since  $\frac{I'}{I} = \frac{2U}{r}$.
\end{proof}

 The next lemma is valid for any function $v$.

\begin{Lem}	\label{l:511}
If $r > \bar{r}>3n$  and  
\begin{align}	\label{e:assum}
	\int_{|x| < \bar{r}} v^2 \, \e^{-f} \leq    \int_{  \bar{r} < |x| < r} v^2 \, \e^{-f} \, ,
\end{align}
then  
\begin{align}
	 \int_{   |x| < r} v^2 \, \e^{-f} \leq \frac{32}{ \bar{r}^2  - 8n}  \, \int_{|x| < r} |\nabla v|^2 \, \e^{-f}   \, .
\end{align}
\end{Lem}

\begin{proof}
Since $\cL \, |x|^2 = 2k - |x|^2$, we have $\dv_f \, \left( v^2 \, x \right) = 2 \, v \langle x , \nabla v \rangle +v^2 ( k - |x|^2/2)$.  Using the absorbing inequality 
$ 2 \, |v \langle x , \nabla v \rangle| \leq |x|^2 v^2/4 + 4 \, |\nabla v|^2$, the divergence theorem gives
\begin{align}
	 \int_{ |x| < r} \left( \frac{|x|^2}{2} - k \right) v^2 \, \e^{-f} &= - r \, \e^{  - \frac{r^2}{4}} \, \int_{|x| =r} v^2 + 2 \int_{|x| < r} v \, \langle x , \nabla v \rangle \, \e^{-f}  \notag \\
		&\leq  - r \, \e^{  - \frac{r^2}{4}} \, \int_{|x| =r} v^2 + \int_{|x| < r} \left( v^2 \, \frac{|x|^2}{4} + 4\, |\nabla v|^2 \right) \,  \, \e^{-f}   \, .
\end{align}
It follows that
\begin{align}
	\left( \frac{\bar{r}^2}{4} - k \right) \,  \int_{\bar{r} < |x|< r}  v^2 \, \e^{-f}  -k\, \int_{|x|< {\bar{r}}}   v^2 \, \e^{-f} 
	\leq \int_{|x|< r} \left( \frac{|x|^2}{4} - k \right) v^2 \, \e^{-f} \leq 4 \,  \int_{|x|< r} |\nabla  v|^2 \, \e^{-f}   \, .  \notag
\end{align}
Bringing in the assumption \eqr{e:assum} gives
\begin{align}
	\left( \frac{\bar{r}^2}{4} - 2k \right) \,  \int_{\bar{r} < |x|< r}  v^2 \, \e^{-f}   \leq 4 \,  \int_{|x|< r} |\nabla  v|^2 \, \e^{-f}   \, .  \notag
\end{align}
The lemma follows from this and using the assumption again.
 \end{proof}

\begin{proof}[Proof of Theorem \ref{t:app}]
We can assume that
$	 \| \phi \|_{L^2}^2 \leq \Lambda^2 \, \int_{|x|<{4n} } v^2 \e^{-f} $
since  otherwise we get \eqr{e:goalA} immediately.
Therefore, given any $r \geq 4n$,    \eqr{e:oldONE}   gives
\begin{align}	\label{e:from1}
 	\left| D - E \right| (r) &\leq 
	 r^{2-k}  \,   \e^{ \frac{r^2}{4} } \, \int_{ |x| < r}  \left(  ( \epsilon + \Lambda ) v^2 + \epsilon |\nabla v|^2   \right) \, \e^{-f} 
	 \, .
\end{align}

Suppose now that some $r \geq 4n$ satisfies
\begin{itemize}
\item[($\star 1$)]   $\int_{|x| < r} v^2 \, \e^{-f} \leq \frac{1}{2} \, \int_{|x| < r} |\nabla v|^2 \, \e^{-f}$.
\end{itemize}
We will use ($\star 1$) to show that $D(r)$ and $E(r)$ are comparable, get  a differential inequality for $U_E (r)$ and    bound  the ratio of the derivatives of quantities in
($\star 1$).  Namely,  $(\star 1$) gives  
\begin{align}	\label{e:e518}
	 \frac{1}{2} \, r^{2-k} \, \e^{ \frac{r^2}{4} } \int_{|x|<r} |\nabla v|^2 \, \e^{-f} \leq  r^{2-k} \, \e^{ \frac{r^2}{4} } \int_{|x|<r}\left(  |\nabla v|^2 - v^2 \right) \, \e^{-f} 
	 = E(r)  \, .
\end{align}
Similarly, using \eqr{e:from1},   $(\star 1$)  and \eqr{e:e518}  gives that
\begin{align}
	|D-E|(r) \leq \left( \frac{3\epsilon}{2} + \frac{\Lambda}{2} \right) \, r^{2-k} \, \e^{ \frac{r^2}{4} } \int_{|x|<r} |\nabla v|^2 \, \e^{-f} \leq \left( 3\epsilon +  \Lambda  \right)  \, E(r)   \, .
\end{align}
We conclude that $   D(r) $, and thus also $I'(r)$,  are also positive and
\begin{align}  \label{e:step1a}
	\left| U - U_E \right| (r) &\leq \left( 3\epsilon +  \Lambda \right)  \, U_E (r) \,  , \\
	\left(1 - 3\epsilon -  \Lambda \right) \, U_E &\leq U(r)  \leq \left(1 + 3\epsilon +  \Lambda \right) \, U_E \, . \label{e:step1b}
\end{align}
Using this in Lemma \ref{l:diffineq} gives the differential inequality at $r$
\begin{align}
	\left( \log U_E \right)'    &\geq 
	\frac{2-k}{r}   + \frac{r}{2}   -   \frac{r}{U_E} -  \left( 1 + 3\epsilon +  \Lambda   \right)^2 \, \frac{U_E}{r}   \, . \label{e:step1}
\end{align}
 From \eqr{e:step1b}, the definition of $U$, and the Cauchy-Schwarz inequality, we get at $r$ that
\begin{align}	\label{e:goodfor1}
	\left(1 - 3\epsilon -  \Lambda \right)^2 \, U_E^2 \, I^2 \leq U^2 \, I^2 = D^2 \leq I \, r^{3-k} \, \int_{|x| =r} |\nabla v|^2 \, .
\end{align}
Noting that $3\epsilon + \Lambda < \frac{1}{2}$, we get
\begin{align}	\label{e:goodfor1a}
	\frac{U_E^2(r)}{4\, r^2}  \leq \left(1 - 3\epsilon -  \Lambda \right)^2 \, \frac{U_E^2(r)}{r^2}  \leq \frac{ \int_{|x| =r} |\nabla v|^2}{ \int_{|x| =r}  v^2 }
	 \, .
\end{align}

We will also need a second property (the first part is the strict form of ($\star 1$)):
\begin{itemize}
\item[($\star 2$)]   $\int_{|x| < r} v^2 \, \e^{-f} < \frac{1}{2} \, \int_{|x| < r} |\nabla v|^2 \, \e^{-f}$ and $ \frac{r^2}{32}  < U_E (r)  $.
\end{itemize}
Set $r_0 = 4n +64\sqrt{2}$.
We will   show that   if   ($\star 2$) holds for some $r \geq r_0$, then it holds for all $s \geq r$.  We will argue by contradiction, so suppose  that $s > r$ is the first time     ($\star 2$) fails.  
Note that ($\star 2$)  is equivalent to  $ U_E(r) > \frac{r^2}{32}   $ and $F(r) > 0$ where 
$
	F(r) = \int_{|x| < r} \left( \frac{1}{2} \, |\nabla v|^2 - v^2 \right) \, \e^{-f}  .
$
Since $s$ is the first time, we   have $F(s) \geq 0$ (i.e., ($\star 1$)), $U_E(s) - \frac{s^2}{32} \geq 0$ and
\begin{align}	\label{e:firsttime}
	F(t) > 0 {\text{ and }} U_E(t) - \frac{t^2}{32} > 0 {\text{ for all }} t \in [r,s) \, .
\end{align}
We also have that at least one of $F(s)$ and $U_E(s) - \frac{s^2}{32}$ is zero.  Suppose first that $F(s)=0$ and, thus, $F' (s) \leq 0$.  However, \eqr{e:goodfor1a} and 
$U_E(s) - \frac{s^2}{32} \geq 0$ give that
\begin{align}	 
	\left( \frac{s}{64} \right)^2 \leq \frac{U_E^2(s)}{4\, s^2}  \leq  \frac{ \int_{|x| =s} |\nabla v|^2}{ \int_{|x| =s}  v^2 }
	 \, .
\end{align}
However, this implies that $F'(s) > 0$ as long as $s > 64 \sqrt{2}$, giving the desired contradiction in the first case.  Suppose now that $U_E(s) = \frac{s^2}{32}$ and, thus, 
$U_E'(s) \leq  \frac{s }{16}$ and
\begin{align}	\label{e:UEs}
	( \log U_E )'(s) \leq \frac{2}{s} \, .
\end{align}
On the other hand, \eqr{e:step1} gives that
\begin{align}
	\left( \log U_E \right)'  (s)   &\geq 
	\frac{2-k}{s}   + \frac{s}{2}   -   \frac{32}{s} -  \left( 1 + 3\epsilon +  \Lambda   \right)^2 \, \frac{s}{32}  \geq \frac{3s}{8} - \frac{k+30}{s} \,  .
\end{align}
This contradicts \eqr{e:UEs} since $s \geq r_0$, completing the proof of the claim.

We will now show that ($\star 2$) holds for $r_1$.  Using the first part of \eqr{e:holefilledA}, we can apply Lemma \ref{l:511} (with $\bar{r} =9n$) to get
\begin{align}
	 \int_{   |x| < r_1} v^2 \, \e^{-f} \leq \frac{32}{ 81n^2  - 8n}  \, \int_{|x| < r_1} |\nabla v|^2 \, \e^{-f} < \frac{1}{2} \,   \int_{|x| < r_1} |\nabla v|^2 \, \e^{-f} \, ,
\end{align}
where the last inequality used that $  81n^2 > 8n + 64$.  This gives the first part of ($\star 2$); in particular, ($\star 1$) holds and  \eqr{e:step1b} gives that
\begin{align}
	U(r_1) \leq \left( 1 + 3\epsilon +  \Lambda \right) \, U_E(r_1) \leq \frac{3}{2} \, U_E(r_1) \, .
\end{align}
Since $\frac{r^2}{16} \leq U (r_1)$ by the second part of \eqr{e:holefilledA},   the second part of ($\star 2$) also holds.

We have established that ($\star 2$) holds for all $r  \geq r_1$,  so we get the differential inequality \eqr{e:step1} for $U_E$ and the equivalence \eqr{e:step1a} 
between $U$ and $U_E$.  This will give the desired growth of $U$ and, thus also, $I$.  We do this next.
  Set $\kappa = \left(  3\epsilon +  \Lambda  \right)$.
We claim that there exists $\bar{R}= \bar{R} (k, r_1) \geq r_1$ so that for all $r \geq \bar{R}$ we have 
\begin{align}	\label{e:thresh}
	U_E (r) > \frac{ r^2 -2k- 68}{2(1+\kappa)^2} \, .
\end{align}
The key  is that if \eqr{e:thresh} fails for some $r \geq r_1$, then \eqr{e:step1} implies that
\begin{align}	 	\label{e:maxpU}
	\left( \log U_E \right)'    \geq \frac{r}{2} - \frac{k+30}{r} - (1+\kappa)^2 \, \frac{U_E}{r} \geq 
	\frac{4}{r}         \, .
\end{align}
On the other hand, for $r \geq 4k$, we have
\begin{align}	\label{e:loggb}
	\left( \log \frac{ r^2 -2k- 68}{2(1+\kappa)^2}  \right)' = \frac{2r}{r^2-2k-68} < \frac{3}{r} 
	\, ,
\end{align}
where the last inequality used that 
$ 
	6k + 204 < r_0^2$.
	Integrating \eqr{e:maxpU} and \eqr{e:loggb} and using  that $  U_E \geq \frac{r^2}{32}$, gives an upper bound for the maximal interval where \eqr{e:thresh} fails. The first derivative test, \eqr{e:maxpU}, and \eqr{e:loggb}  imply that once \eqr{e:thresh} holds for some $R \geq r_1$, then it also holds for all   $r \geq R$.  This gives the claim.
Using \eqr{e:step1a} and \eqr{e:thresh}, we get for $r \geq \bar{R}$ that
\begin{align}	\label{e:thresh2}
	U (r) \geq  (1-\kappa) U_E (r) >    \frac{ (1-\kappa)  }{(1+\kappa)^2} \left( \frac{r^2}{2} - k -34\right)  \, .
\end{align}
Integrating this from $\bar{R}$ to $R$ gives that
\begin{align}
	\log \frac{I(R)}{I(\bar{R})} = 2\, \int_{\bar{R}}^R  \frac{U(r)}{r} \, dr &\geq   \frac{ (1-\kappa)  }{(1+\kappa)^2} \, \left( \frac{R^2 - \bar{R}^2}{2} -  (2k+68)  \, \log \frac{R}{\bar{R}} \right) \, .
\end{align}
Since    $\bar{R}$ depends only on $k$ and $r_1$, exponentiating gives $C=C(k,r_1)$ so that
\begin{align}	\label{e:521}
	\sup_{r_1 \leq r \leq \bar{R} } \, \, I(r) = I(\bar{R}) \leq C \, I(R) \, R^{(2k+68)\,  \frac{1-\kappa}{(1+\kappa)^2}} \, \e^{ - \frac{(1-\kappa)}{2(1+\kappa)^2}\, R^2 }   \, .
\end{align}
Choose $r_2 \in [r_1 , 2r_1]$ that achieves the minimum of $D$ on $[r_1 , 2r_1]$.  Since $I' = \frac{2D}{r}$, it follows that $D(r_2) \leq I (2r_1)$.
Therefore, since ($\star 1$) holds for $r_2$,  we have
\begin{align}	\label{e:537}
	r_2^{2-k} \, \e^{ \frac{r_2^2}{4} } \, \int_{|x| < r_2}  v^2 \, \e^{f} \leq E(r_2) \leq \frac{D(r_2)}{(1-\kappa)} \leq 2 \, I(2r_1) \, .
\end{align}
Finally, combining  \eqr{e:521} and \eqr{e:537} gives \eqr{e:goalA}.
\end{proof}

\section{General frequency}

In this section, we will  prove  Theorem \ref{t:goalCM7} by showing that either we already
   have the  bound on   $w$ or    \eqr{e:holefilledA} holds
and  Theorem \ref{t:app}   bounds $w$.   
Throughout this section, we will assume that $w$ satisfies \eqr{e:oldONE}--\eqr{e:wijbounds} and (A1)--(A3).

 \vskip2mm
The main task left is to prove the following proposition:

 \begin{Pro}	\label{l:getbeta}
 Given $r_+ \geq 9n$, there exist $\lambda > r_+$ and $\zeta_2$ so that if  $\zeta \geq \zeta_2$ and
 \begin{align}	\label{e:lambda}
 	   \int_{|x|<4n} \left| \Hx_w \right|^2 \geq \zeta  \, \mu^2 \, ,
\end{align}
then  there exists $r_1 \in ( r_+,  \lambda  )$ satisfying \eqr{e:holefilledA}.
  \end{Pro}

\vskip1mm  Throughout this section, $C_r$ will be a constant that depends on $r$ (but not on $w$ or $\mu$) that will be allowed to change from line to line.

\subsection{Proof of Theorem \ref{t:goalCM7} assuming Proposition \ref{l:getbeta}}

\begin{Lem}	\label{l:avgs}
Given $r> 0$, there exists $C_r$ so   that
\begin{align}	\label{e:muCr}
	  \left|  \int_{|x| < r} \Hx_w \right|  +
	  \left|  \int_{|x| <r}  w \right| &\leq C_r \, \mu \, .   
\end{align}
Furthermore, given any $h$ on $\SS^{n-k}$ with $\Delta_{\theta} h = - \frac{1}{2} \, h$ and $\int h^2 (\theta) \, d\theta = 1$, we have
\begin{align}
	 \left|  \int_{|x| < r}  h\, \nabla^{\RR^k} w \right| \leq C_r \, \mu \, .
\end{align}
\end{Lem}

\begin{proof}
Let $w_{ij}$ be a Euclidean second derivative and 
define  the spherical average
\begin{align}
	J_{ij} (r) = r^{1-k} \, \int_{|x| =r} w_{ij} \, .
\end{align}
By (A1), we have  $J_{ij}   (0) = 0$.  Note that $\left| \cL \, w_{ij} \right| \leq C_r \, \mu  $, so we have
\begin{align}
	|J_{ij}' | (r) \leq r^{1-k} \, \e^{ \frac{r^2}{4} } \int_{|x| < r} \left|  \cL \, w_{ij} \right| \, \e^{-f}  \leq C_{r} \, \mu \, .
\end{align}
 Thus,  we get that $|J_{ij}(r)| \leq C_{r} \, \mu$.
 Integrating this gives the integral bound on $\Hess_w$ in \eqr{e:muCr} and, thus, the same bound on $\left| \int_{|x| < r} \Delta_{\RR^k} w \right|$.   The bound on
  $\left| \int_{|x| < r}  w \right|$ follows similarly by setting $J (r) = r^{1-k} \, \int_{|x| =r} w$.  Namely,  (A3) bounds $J(0)$ and we bound $J'(r)$ by using that
$\Delta_{\theta} w$ integrates to zero over each sphere and $\left| \int_{|x| < r} \Delta_{\RR^k} w \right| \leq C_r \, \mu$.

To get the last claim, define a vector-valued function $J_h (r)$ by  
\begin{align}
	J_h (r) = r^{1-k} \,    \int_{|x| = r}  h\, \nabla^{\RR^k} w      \, , 
\end{align}
so that $J_h (0) = 0$ by  (A2).  Arguing as above and using the integral bound on the Euclidean Hessian bounds $J_h(r)$ and integrating this gives the last claim.
\end{proof}

\begin{Cor}	\label{p:propw}
Given $\bar{r} > 4n$,  there exist $C_{\bar{r}}$ so that 
\begin{align} \label{e:A16}
	\int_{  |x| <\bar{r}} \left| \Hx_w \right|^2 \leq C_{\bar{r}} \, \mu^2 + C_{\bar{r}} \, \int_{\bar{r} < |x| < \bar{r} +1} \left| \Hx_w \right|^2   \, .
\end{align}
 \end{Cor}

\begin{proof} 
Set $\cA = \{\bar{r} < |x| < \bar{r} +1\}$. Let
 $w_{ij}$ be  a Euclidean second derivative and $\eta$ a cutoff function that is one for $|x| < \bar{r}$, zero for $\bar{r} +1 < |x|$, and   $|\nabla \eta| \leq 2$.  Given  $\delta > 0$, we get  
\begin{align}
	\dv_f \, \left( \eta^2 w_{ij} \, \nabla w_{ij} \right) &= \eta^2 \left( |\nabla w_{ij} |^2 + w_{ij} \, \cL w_{ij} \right) + 2\eta    \, w_{ij} \langle \nabla  w_{ij} , \nabla \eta \rangle \notag \\
	&\geq \eta^2 \left(\frac{1}{2} \,  |\nabla w_{ij} |^2 - \delta \, w_{ij}^2 - \frac{1}{4\delta} \left| \cL w_{ij} \right|^2\right) - 2|\nabla \eta|^2     \, w_{ij}^2     \, .
\end{align}
We get that
\begin{align}	\label{e:510}
	 \int_{ |x| < \bar{r} }   |\nabla w_{ij} |^2 \, \e^{-f} &\leq   2\,  \delta \, \int_{|x| < \bar{r} } w_{ij}^2 \, \e^{-f}+  \frac{1}{2\delta} \int_{|x| < \bar{r} +1} \left| \cL w_{ij} \right|^2 \, \e^{-f}  +  (8+ 2\delta) \, \int_{\cA}  w_{ij}^2  \, \e^{-f}   \notag  \\
	 &\leq 2\,  \delta \, \int_{|x| < \bar{r} } w_{ij}^2 \, \e^{-f}+ C_{\bar{r}} \frac{\mu^2}{\delta}   +  (8+ 2\delta) \, \int_{ \cA}  w_{ij}^2  \, \e^{-f}   \, .  
\end{align}
On the other hand, Lemma \ref{l:avgs} and the Neumann Poincar\'e inequality give $C_{\bar{r}}$ so that
\begin{align}
	  \int_{|x| < \bar{r} } w_{ij}^2 \, \e^{-f} \leq C_{\bar{r}} \, \left( \mu^2 +  \int_{ |x| < \bar{r} }   |\nabla w_{ij} |^2 \, \e^{-f}  \right) \, .
\end{align}
Using this to bound the first term on the right in \eqr{e:510} and  taking $\delta > 0$ small enough (depending on $\bar{r}$),  this can be absorbed.  Finally,  summing over $i,j$  gives the corollary.
\end{proof}

\begin{proof}[Proof of Theorem \ref{t:goalCM7}]
  Lemma \ref{l:taylorsphere} gives   $\tilde{w}$ as in \eqr{e:kernelcC} so that $v = w - \tilde{w}$ 
  satisfies \eqr{e:oldONE}--\eqr{e:wijbounds} and (A1), (A2) and (A3).  By Corollary \ref{c:bartime}, it suffices to get $\int_{|x| < 4n} v^2 \leq C \, \mu^{\beta}$ with  $\nu < \beta$.

Proposition \ref{l:getbeta} gives $\lambda$ and $\zeta_2$ (depending just on $n$) so that if  \eqr{e:lambda} holds with $\zeta \geq \zeta_2$, then there exists $r_1$
satisfying \eqr{e:holefilledA} with
\begin{align}
	r_1 \in ( \max \{ 9n , \sqrt{8n+256} \} , \lambda ) \, .
\end{align}
We can assume that  \eqr{e:lambda} holds with $\zeta \geq \zeta_2$ since the theorem otherwise follows from Corollary \ref{c:bartime}.    Therefore, Theorem \ref{t:app}
applies and we get 
 $\bar{R} = \bar{R}(n,r_1)$ and  $C = C(n,r_1)$  so that   for any $\Lambda \in (0, 1/3)$
\begin{align}	\label{e:goalAA}
	\int_{|x| < {4n}} v^2 \, \e^{-f}  &\leq   \Lambda^{-2} \, \| \phi \|_{L^2}^2 + C \, I(R) \, R^{2n+68} \, \e^{ - \frac{ \left(1-3\epsilon-  \Lambda  \right)\, R^2}{2\left(1+3\epsilon+  \Lambda  \right)^2}}  \notag \\
	&\leq   \Lambda^{-2} \, \mu^2 + C   \, R^{2n+68} \, \left( \e^{ - \frac{R^2}{2} }  \right)^{ \frac{ \left(1-3\epsilon-  \Lambda  \right) }{ \left(1+3\epsilon+  \Lambda  \right)^2}} \, .
\end{align}
This required  $R \geq \bar{R}$; if $\bar{R} > R$, then   there is a positive lower bound for $\mu$ and the theorem holds trivially. 
Since $\e^{ - \frac{R^2}{2}} \leq \mu^2$, the theorem follows by taking  $\epsilon , \Lambda > 0$ small enough   that
\begin{align}
	  \left(1-3\epsilon- \Lambda \right) > \nu \, \left(1+3\epsilon+ \Lambda  \right)^2   \, .
\end{align}
\end{proof}

\subsection{Proof of Proposition \ref{l:getbeta}}

We will  get a positive lower bound for the frequency $U_2$ for $\Hx_w$ that will force $\Hx_w$ to grow very rapidly.  We will then   combine Poincar\'e and reverse Poincar\'e inequalities to show that $w$ itself grows rapidly as claimed.
To do this,  define   quantities $I_2$, $D_2$ and   $U_2$ for $\Hx_w$   by
\begin{align}
	I_2 (r) = r^{1-k} \, \int_{|x| = r} \left| \Hx_w \right|^2 \, ,
\end{align}
  $D_2 = \frac{r}{2} \, I_2'$, and    $U_2 = \frac{D_2}{I_2}$ so that  $\left( \log I_2 \right)'  = \frac{2\, U_2}{r}$.  Define $\psi = \left( \cL + 1 \right) w$ so that
$
	\cL w_{ij} = \psi_{ij}$.
	Differentiating $I_2$, we see that
\begin{align}	\label{e:D2r}
	D_2 (r) &= r^{2-k} \sum_{i,j} \int_{|x|=r} w_{ij} \partial_r w_{ij} =  r^{2-k}\, \e^{ \frac{r^2}{4} } \,   \int_{|x|<r}  \left(  \left| \nabla \Hx_w   \right|^2 + \sum_{i,j} w_{ij} \psi_{ij} \right)\, \e^{-f} \, .
\end{align}

The next  two lemmas give a differential inequality for $U_2$ when $U_2 > 0$ and then establish that $U_2(r)$ is positive on an interval.

\begin{Lem}	\label{l:freqU2}
 If $U_2 (r) > 0$, then 
 \begin{align}
	(\log U_2)'(r) \geq \frac{2-k}{r}  + \frac{r}{2}   -\frac{ U_2}{r}    - \frac{r }{U_2} \, \left(  \frac{ r^{1-k} \int_{|x|=r}    \sum_{i,j}   \psi_{ij}^2  }{I_2}   \right)^{ \frac{1}{2} } \, .
\end{align}
 \end{Lem}
 
 \begin{proof}
 Differentiating $D_2$ gives that
\begin{align}	\label{e:D2rA}
	D_2' (r) &= \frac{2-k}{r} \, D_2 + \frac{r}{2} \, D_2 +   r^{2-k}\,   \int_{|x|=r}  \left(  \left| \nabla \Hx_w   \right|^2 + \sum_{i,j} w_{ij} \psi_{ij} \right)  \, .
\end{align}
The first equality in \eqr{e:D2r} and the Cauchy-Schwarz inequality give  that
\begin{align}	\label{e:581}
	D_2^2(r)   \leq I_2(r) \, r^{3-k} \,  \int_{|x|=r} \sum_{i,j}  ( \partial_r w_{ij} )^2  \leq I_2(r) \, r^{3-k} \,  \int_{|x|=r} \left|\nabla \Hx_w \right|^2 \, .
\end{align}
Since $D_2 (r) > 0$ (by assumption), using \eqr{e:581} in \eqr{e:D2rA} and dividing by $D_2(r)$ gives  
 \begin{align}
	(\log D_2)'(r) \geq \frac{2-k}{r}  + \frac{r}{2}   +\frac{ U_2}{r}    - \frac{r^{2-k}}{D_2} \, \int_{|x|=r} \left|  \sum_{i,j} w_{ij} \psi_{ij} \right| \, .
\end{align}
The lemma follows from this and the Cauchy-Schwartz inequality  since $(\log I_2)' = \frac{2 U_2}{r}$.
 \end{proof}

\begin{Lem}	\label{l:jpos}
Given $\lambda > 4n$, there exists $\zeta_0$  so that if  \eqr{e:lambda} holds for $\zeta \geq \zeta_0$, then for each $r \in (4n , 2\lambda)$ we have $U_2(r) \geq 0$  
 and, moreover, for $r \in (4n , 2\lambda-1)$ there exists $c_r > 0$ so  
 \begin{align}	\label{e:lastbound}
	\max \, \{ U_2 (s) \, | \, s \in [r,r+1] \} > c_r \, .
\end{align}
\end{Lem}

\begin{proof}
Given    $r \in (4n , 2\lambda)$, the Neumann Poincar\'e inequality,  \eqr{e:muCr}  and \eqr{e:lambda} give
\begin{align}
	\int_{|x| < r} \left| \Hx_w \right|^2 \leq   C_r \,  \mu^2 + C_r \,   \int_{|x| < r} |\nabla \Hx_w |^2 \leq \frac{C_r}{\zeta} \int_{|x| < 4n} \left| \Hx_w \right|^2 + C_r \,   \int_{|x| < r} |\nabla  \Hx_w |^2 \, .
		\notag
\end{align}
If $\zeta$ is large enough (depending on $\lambda$), we can absorb the first term on the right to get
\begin{align}	\label{e:zetalarge}
	\int_{|x| < r} \left| \Hx_w \right|^2 \leq     C_r \,   \int_{|x| < r} |\nabla \Hx_w |^2 \, .
\end{align}
To bound the error term in  \eqr{e:D2r}, use the absorbing inequality to get for any $\delta > 0$
\begin{align}
	\sum_{i,j} \, \int_{|x|<r} |\psi_{ij} w_{ij}| \e^{-f} \leq \delta  \int_{|x|<r} \left| \Hx_w \right|^2\e^{-f}  + \frac{C_r\, \mu^2}{\delta}    \, .
\end{align}
Taking $\delta > 0$ small and then $\zeta$ even larger, the last two inequalities give that
\begin{align}
	\sum_{i,j} \, \int_{|x|<r} |\psi_{ij} w_{ij}| \e^{-f} \leq    \frac{1}{2} \int_{|x|<r} |\nabla \Hx_w|^2 \e^{-f} \, .
\end{align}
Using this in  \eqr{e:D2r}, we conclude that
\begin{align}	 \label{e:Dposi}
	D_2 (r) \geq \frac{r^{2-k}}{2} \,    \e^{ \frac{r^2}{4} } \,   \int_{|x|<r}   \left| \nabla \Hx_w   \right|^2  \, \e^{-f} \geq C_r \, \int_{|x| < r} \left| \Hx_w \right|^2 = 
	C_r' \int_0^{r} s^{k-1} \, I_2(s) \, ds    \, .
\end{align}
In particular, $U_2 (r) \geq 0$.  Moreover, we get for $r \in (4n , 2\lambda-1)$ that
 \begin{align}	\label{e:eitherway}
 	D_2 (r+1) \geq    C_r \, I_2 (r) \, .
\end{align}
    Note that Corollary \ref{p:propw} implies that $I_2(r) > 0$.
    We  have either $I_2(r+1) \leq 2 \, I_2 (r)$ or $2\, I_2 (r) <  I _2(r+1)$;  the   claim \eqr{e:lastbound} follows from \eqr{e:eitherway} in each case, completing the proof.
\end{proof}

 The next lemma gives   $r_n$ so that  $U_2 (r) \geq \frac{r^2}{3}$ when $r \geq r_n$ as long as \eqr{e:lambda} holds for a large   $\zeta$ that depends on $\lambda$.  It will be crucial that  $r_n$  does not depend on $\lambda$.

 \begin{Lem}	\label{l:Uc}
Given $\lambda > 4n$, there exists $\zeta_1$  so that if  \eqr{e:lambda} holds for $\zeta \geq \zeta_1$, then for each $r \in (4n , 2\lambda)$ we have $U_2 (r) > 0$ and, moreover,
\begin{align}
	U_2(r) \geq  \frac{r^2}{3}  {\text{ for   $r \in( r_n, 2\lambda)$, where $r_n$ depends only on $n$}} \, .
\end{align}

 \end{Lem}
 
 \begin{proof}
 We will choose $\zeta_1$ even greater than the $\zeta_0$ given by Lemma \ref{l:jpos}.  Thus, 
 Lemma \ref{l:jpos} gives that $I_2'(r) \geq 0$ for  $r \in (4n , 2\lambda)$ and \eqr{e:lastbound} holds.  Let $c_{4n}$ be the constant from \eqr{e:lastbound} with $r=4n$, so that there exists $s \in (4n, 4n+1)$ with
 \begin{align}	\label{e:c4n}
 	U_2 (s) \geq c_{4n} > 0 \, .
\end{align}
   Corollary \ref{p:propw} and the monotonicity of $I_2$ give $C_0$ so that
 $C_0 \, \zeta \, \mu^2 \leq I_2(r)$ and, thus,
 \begin{align}
	 \left(  \frac{ r^{1-k} \int_{|x|=r}    \sum_{i,j}   \psi_{ij}^2  }{I_2}   \right)^{ \frac{1}{2} } \leq \frac{C_{\lambda}}{\zeta} \, .
\end{align}
Using this in Lemma \ref{l:freqU2} gives for $r \in (4n , 2\lambda)$ that
\begin{align}	\label{e:diqA}
	(\log U_2)'(r) \geq \frac{2-k}{r}  + \frac{r}{2}   -\frac{ U_2}{r}    - \frac{C_{\lambda}}{\zeta U_2}  \, .
\end{align}
Now choose $\zeta_1 > \zeta_0$ so that $\frac{C_{\lambda}}{\zeta_1 \, c_{4n}} \leq \frac{1}{4}$. Thus,  if $c_{4n} \leq U_2(r) \leq \frac{11r^2}{30}$ and  $r \in (4n , 2\lambda)$, then
\begin{align}	\label{e:diq}
	(\log U_2)'(r) \geq     \frac{ r^2 - 2\,U_2 (r)}{2r}    - \frac{1}{2} \geq \frac{4r-15}{30} \geq \frac{r}{24} > 0 \, .
\end{align}
Combining this with \eqr{e:c4n}, we see that $U_2 \geq c_{4n}$ for $r \in (4n+1 , 2\lambda)$.  Arguing as in the proof of \eqr{e:thresh},  \eqr{e:diq} gives that
\begin{itemize}
\item  There exists $r_n$ depending  on $n$ so that there is   $r_1 \in [4n+1 , r_n]$ with $U_2 (r_1) > \frac{r_1^2}{3}$.
\item  There cannot be a first $r \in( r_1,2\lambda)$ with $U_2(r) = \frac{r^2}{3}$.
\end{itemize}
 \end{proof}
 
 The next lemma uses reverse Poincar\'e inequalities to bound the Euclidean Hessian in terms of the $L^2$ norm of the function on a larger set.
 
 \begin{Lem}	\label{l:RPo}
 If $r> 4n$ and \eqr{e:lambda} holds for $\zeta \geq 84$, then
  \begin{align}	 	 
	\int_{|x| < r} \left| \Hx_w \right|^2 \, \e^{-f} \leq   204 \, \int_{|x| < r +2} w^2 \, \e^{-f}
	\, .
\end{align}
 \end{Lem}
 
 \begin{proof}
 Given a compactly supported function $\eta (x)$ and a Euclidean partial derivative $w_i$, we have $\left( \cL + \frac{1}{2} \right) w_i  = \psi_i$ and, thus, 
\begin{align}	\label{e:revP}
	\dv_f \left(\eta^2 w_i \nabla w_i \right) = \eta^2 \left( |\nabla w_i|^2 - \frac{1}{2} w_i^2 + w_i \psi_i \right) + 2 \eta w_i \langle \nabla \eta , \nabla w_i \rangle \, .
\end{align}
The divergence theorem and   Cauchy-Schwarz and absorbing  inequalities give 
\begin{align}
	\int \eta^2 \, |\nabla w_i|^2 \, \e^{-f} \leq \int \left( (4 |\nabla \eta |^2 + 2 \eta^2  )w_i^2 + \eta^2 \psi_i^2 \right) \, \e^{-f} \, .
\end{align}
Taking $\eta \leq 1$ identically one for $|x| < r $  and cutting off linearly for  $r   < |x| <r + 1$, we get
\begin{align}	 	 
	\int_{|x| < r} \left|\nabla    w_i \right|^2 \, \e^{-f} \leq        \int_{|x| < r +1}    |  \psi_i |^2 \, \e^{-f} + 6\, \int_{|x| < r +1} \left|   w_i \right|^2 \, \e^{-f}
	\, .
\end{align}
Since  \eqr{e:thenewguy} gives that
$\left|   \psi_i \right|  \leq |\nabla \phi|  + \epsilon \, \left( |w| + |\nabla w| + |\nabla w_i| \right)  $  and   $\| \nabla \phi \|_{L^2} \leq \mu$, we get
\begin{align}	 	\label{e:firstRP}
	\int_{|x| < r} \left| \Hx_w \right|^2 \, \e^{-f} \leq   2\, \mu^2 +       \int_{|x| < r +1} \left( 2 w^2 + 10\left| \nabla w \right|^2 \right) \, \e^{-f}
	\, .
\end{align}

We will now argue similarly to bound the right-hand side of \eqr{e:firstRP}
 in terms of $w$ itself.  We will again let $\eta$ be a cutoff function (on a different set).  We have
\begin{align}
	\dv_f \, \left( \eta^2 w \nabla w \right) = \eta^2 \, \left(  |\nabla w|^2 - w^2 + w \psi \right) + 2\eta w \, \langle \nabla w , \nabla \eta \rangle \, .
\end{align}
Using the absorbing inequality $|2\eta w \, \langle \nabla w , \nabla \eta \rangle | \leq   \eta^2 |\nabla w|^2/2 + 2 |\nabla \eta  |^2 w^2$ and the Cauchy-Schwarz inequality on the $w\psi$ term, the divergence theorem gives that
\begin{align}
	\int \eta^2 \,|\nabla w|^2 \, \e^{-f} \leq \int (3\eta^2 +4 |\nabla \eta |^2) \,  w^2 \, \e^{-f} +   \| \eta \, \psi \|_{L^2}^2 \, .
\end{align}	
Equation \eqr{e:oldONE} gives that $\psi^2 \leq 2\phi^2 + 2\epsilon (w^2 + |\nabla w|^2)$, so we get
\begin{align}	 
	\int \eta^2 \left| \nabla w \right|^2 \, \e^{-f} \leq    2\, \mu^2 + 2\epsilon \, \int \eta^2  |\nabla w|^2  \, \e^{-f}  +   \int  \left( (3+2\epsilon) w^2 + |\nabla \eta|^2 \right)  w^2 \, \e^{-f}
	\, .  
\end{align}
Since $\epsilon < \frac{1}{9}$, we can absorbe the $|\nabla w|^2$ term.  Thus, 
taking $\eta \leq 1$ identically one for $|x| < r+1 $  and cutting off linearly for  $r+ 1 < |x| < r+2$, we get
\begin{align}	 
	\int_{|x| < r +1}  \left| \nabla w \right|^2 \, \e^{-f} \leq    4 \, \mu^2 + 10 \, \int_{|x| < r +2}   w^2 \, \e^{-f}
	\, .
\end{align}
Combining this with \eqr{e:firstRP} gives that
\begin{align}	 	 
	\int_{|x| < r} \left| \Hx_w \right|^2 \, \e^{-f} \leq   42 \,\mu^2 + 102 \, \int_{|x| < r +2} w^2 \, \e^{-f}
	\, ,
\end{align}
 The lemma follows since \eqr{e:lambda} implies that $42 \, \mu^2 \leq \frac{1}{2} \, \int_{|x| < 4n} \left| \Hx_w \right|^2 \, \e^{-f}$.
 \end{proof}
 
   Given a  function $u$ on the cylinder, $y \in \RR^k$,  and $\lambda \in \RR$, let $\Psi_{\lambda , u,y}$ be the norm squared of the projection of $u$ on  the $\lambda$ eigenspace of  $\Delta_{\theta}$ on the sphere $x=y$.  Let $B_R^k$ be the ball in $\RR^k$.

 \begin{Lem}	\label{l:eigP}
  Given   $\lambda \in \RR$, there exists $C$ depending on $\lambda , k , n$ so that
  \begin{align}
  	\int_{|x| < R} u^2 & \leq \int_{B_R^k} \Psi_{\lambda , u,x}   + C  \,  R^2 \,  \int_{|x| < R} |\nabla^{\RR^k}  u |^2   + C \, \int_{|x| < R} \left\{ \left( (\cL + \lambda )u \right)^2 +  \left|  \Hx_u \right|^2 \right\} 
  \end{align}
 \end{Lem}
 
 \begin{proof}
Since   $\cL + \lambda = (\Delta_{\theta} + \lambda ) + \cL_{\RR^k}$, we get for each $y \in B_R^k$ that
 \begin{align}
 	\int_{x=y} u^2 &\leq  \Psi_{\lambda , u,y} + c \, \int_{x=y} \left( (\Delta_{\theta} + \lambda )u \right)^2 \leq  \Psi_{\lambda , u,y} + 2\, c \, \int_{x=y} \left( (\cL + \lambda )u \right)^2 + 2 \, c \, \int_{x=y} \left(  \cL_{\RR^k} u \right)^2
	\notag \\
	&\leq   \Psi_{\lambda , u,y}  + 2\, c \, \int_{x=y} \left( (\cL + \lambda )u \right)^2 + 4 \, k\, c \, \int_{x=y} \left|  \Hx_u \right|^2 +  4 \, c \, R^2 \, \int_{x=y} \left|  \nabla^{\RR^k} u \right|^2
	 \, .
 \end{align}
Integrating this over $B_R^k$ gives the lemma.
 \end{proof}

 The next lemma is a Poincar\'e inequality bounding $w$ by $\Hx_w$.
 
 \begin{Lem}	\label{l:PIq}
 We have 
 \begin{align}
 	\int_{|x| < R} w^2 & \leq    C_R \, \int_{|x| < R+ 1}  \left|  \Hx_w \right|^2 + C_R \, \mu^2
 \end{align}
 \end{Lem}
 
 \begin{proof}
  Lemma \ref{l:eigP} with $u=w$ and $\lambda = 1$, so that  $\Psi_{\lambda , u,x}\equiv 0$ and $|(\cL + 1 )w| \leq C_r \, \mu$, gives 
   \begin{align}
  	\int_{|x| < R} w^2 & \leq  C  \,  R^2 \,  \int_{|x| < R} |\nabla^{\RR^k}  w |^2   + C \, \int_{|x| < R}  \left|  \Hx_w \right|^2 + C_R \, \mu^2 \, .
  \end{align}
 We need to absorb the first term on the right side.  Let $w_i$ be a Euclidean derivative of $w$.  Applying  Lemma \ref{l:eigP} with $u=w_i$ and $\lambda = \frac{1}{2}$ gives
   \begin{align}
  	\int_{|x| < R} w_i^2 & \leq   \int_{B_R^k} \Psi_{\frac{1}{2} , w_i,x}  + C  \,  R^2 \,  \int_{|x| < R} |\Hx_{w} |^2   + C \, \int_{|x| < R}  \left|  \Hx_{w_i} \right|^2 + C_R \, \mu^2 \, .
  \end{align}
  Let $\{ h_j \}$ be an $L^2$-orthonormal basis of $\frac{1}{2}$-eigenfunctions for $\Delta_{\theta}$ and define  $\Psi_{ij} (x) =
 \int h_j (\theta) w_i (x,\theta) \, d\theta$.  It follows that 
\begin{align}
	\Psi_{\frac{1}{2} , w_i,x} =\sum_j \left( \Psi_{ij} (x)  \right)^2 \, .
\end{align}
By Lemma \ref{l:avgs}, 
$
 	\left( \int_{|x| < r} \Psi_{ij} (x) \right)^2 \leq C_r \, \mu^2$.
	 The  Poincar\'e inequality on $B_R^k$ gives 
 \begin{align}
 	\int_{B_R^k}  \left( \Psi_{ij} (x)  \right)^2 \leq C_r \, \mu^2  + C_R \, \int_{B_R^k} \left| \nabla^{\RR^k} \Psi_{ij} \right|^2
	\leq C_r \, \mu^2  + C_R  \, \int_{|x| < R} \left| \Hx_w \right|^2 \, .
 \end{align}
 Putting this together gives 
  \begin{align}
 	\int_{|x| < R} w^2 & \leq  C_R \,  \int_{|x| < R} |\nabla^{\RR^k}  \Hx_w |^2   + C_R \, \int_{|x| < R}  \left|  \Hx_w \right|^2 + C_R \, \mu^2 \, .
 \end{align}
Finally, to complete the proof, we use $|\cL \Hx_w| \leq C_r \, \mu$ and the reverse Poincar\'e inequality  to bound the $|\nabla \Hx_w|$ term.
 \end{proof}

 \begin{proof}[Proof of Proposition \ref{l:getbeta}]
 We will fix $\lambda$ at the end depending just on $n$ and $r_+$ and then choose $\zeta_2$.
 Given $r> 4n$, Lemma \ref{l:PIq} and  \eqr{e:lambda}
give 
 \begin{align}	\label{e:startE1}
 	\int_{|x| < r} w^2 & \leq    C_r' \, \int_{|x| < r+1}  \left|  \Hx_w \right|^2 + C_r \, \mu^2 \leq    C_r \, \int_{4n < |x| < r+1}  \left|  \Hx_w \right|^2 \, .
 \end{align}
Let $r_n$ be given by Lemma \ref{l:Uc}, so that $U_2 (r) \geq \frac{r^2}{3} $ for $r \geq r_n$.  If $r_n \leq r$ and $r+1 \leq s \leq \lambda -2$, then  \eqr{e:startE1} and Lemma \ref{l:Uc} give that
\begin{align}	 	 
	  \int_{|x| < { {r}}} w^{2}   &   \leq  
	    C_{  {r}} \, I_2 (r+1)   \leq 	  C_r \, 
	     I_2 (s) \, \e^{ - \frac{s^2}{3} }  	  \, .  \label{e:U2grows}
\end{align}
On the other hand, Lemma \ref{l:RPo}
gives that
  \begin{align}	 	 	\label{e:RPo}
	\int_{0}^s t^{k-1} \, I_2(t) \, \e^{ - \frac{t^2}{4} } \, dt = \int_{|x| < s} \left| \Hx_w \right|^2 \, \e^{-f} \leq   204 \, \int_{|x| < s +2} w^2 \, \e^{-f}
	\, .
\end{align}
 It follows that we can choose $r_a$, depending just on $n$, so that 
 $
 	2 \, \int_{|x| < 9n} w^2 \, \e^{-f} \leq \int_{|x| < r_a} w^2 \, \e^{-f} $.
	 In particular, the first part of \eqr{e:holefilledA} holds for any $r_1 \geq r_a$.  
 
Let $I_0$ and $U_0$ be the quantities $I$ and $U$ for $w$.  We repeat the argument   starting from $r=\max \{ r_a , r_+ \}$   using $U_2 \geq \frac{r^2}{3}$ to force $I_0$ to grow. For $\lambda$ large,  depending on $n$ and $r_+$, this gives $r_1 \in (\max \{ r_a , r_+ \} , \lambda)$ with $U_0 (r_1) \geq \frac{r^2}{16}$ (we could do this for any rate below $\frac{r^2}{3}$).  Finally,  choose $\zeta_2 > 84$ larger than the $\zeta_1$ from Lemma \ref{l:Uc} with this $\lambda$.
   \end{proof}

  \section{Proving the estimate for rescaled MCF}		\label{s:s4}
 
   We will now prove      \eqr{e:rmcfa}  and, thus, complete the proof of Theorem \ref{t:thom}.
From now on,  $\Sigma_t \subset \RR^{n+1}$ is a rescaled MCF with  $\lambda (\Sigma_t)   < \infty$
and  $\Sigma_t$ converges as $t \to \infty$ to a cylinder $\cC = \SS^{n-k}_{\sqrt{2(n-k)}} \times \RR^k$.  The sequence $\delta_j$ is defined in \eqr{e:deltaj}.

 \begin{Pro}	\label{t:rMCF}
 \eqr{e:rmcfa} holds.
 \end{Pro}

        \begin{proof}
         We will assume that $k \geq 1$ as 
  the case $k=0$ follows similarly, but much more easily.
     Let $\bar{\beta} < 1 $ be given by Proposition \ref{c:djsums1}. The proposition will follow once we show  
   \begin{align}	 
		\sup_{t \in [j,j+1]} \, \, \sup_{B_{2n}\cap \Sigma_t} \, \left| \Pi_{j+1} (\nabla \bar{H}) \right| &\leq C \, \delta_j^{ \bar{\beta}}  \, , \label{e:show1}  
		\end{align}
	where $C$ does not depend on $j$.  
	
	We  next explain how the parameters will be chosen.  First, since $\bar{\beta} < 1$, we can choose $\nu , \beta < 1$ so that $\bar{\beta} < \nu \, \beta$.  Next, given this $\nu$,  Theorem \ref{t:goalCM7} gives   $\bar{\epsilon} > 0$.  Finally, we choose  
	  the constant $\epsilon_1 > 0$  in Proposition \ref{p:evolving} to 
	ensure that  \eqr{e:oldONE}  holds with $\bar{\epsilon}$.

 Proposition \ref{p:evolving} with $\epsilon_1$  and   $\beta  \in (\bar{\beta} , 1)$ as above  gives
                      constants $R_j , C$ and 
                   cylinders $\cC_j$ so that $B_{R_j} \cap \Sigma_t$ is a graph over $\cC_{j+1}$ of a function $w$ for each $t \in [j,j+1]$.  Moreover, (1), (2) and (4) in  Proposition \ref{p:evolving} 
                       give \eqr{e:oldONE}--\eqr{e:wijbounds} with $\epsilon < \bar{\epsilon}$ and 
                     $
                       	\mu = C \, \delta_j^{ \beta }$.
	                        	 		  Theorem \ref{t:goalCM7} now applies with our choice of $\nu < 1$ above.  Thus,  we get a constant $C_{\nu}$ and function
\begin{align}	 
	\tilde{w} = \sum_i  a_i (x_i^2 - 2) + \sum_{i<k} a_{ik} \, x_i x_k + \sum_k   x_k h_k (\theta) \, , 
\end{align}
where $a_i , a_{ik}  $ are constants and each $h_k(\theta)$ is a   $\frac{1}{2}$-eigenfunction for $\Delta_{\theta}$, 
and we have
\begin{align}	 \label{e:goodap2}
	\sup_{|x| \leq 3n} \, \left| w - \tilde{w} \right| \leq C_{\nu} \,\mu^{ \nu } = C \, \delta_j^{\beta \, \nu}   \, .
\end{align}
The $\RR^k$ unit vector fields $\partial_{x_i}$ on $\cC_{j+1}$ push forward to vector fields on    $\Sigma_t$   that we   still denote  $\partial_{x_i}$.  	Since $\{ |x| \leq 3n\}  \cap \Sigma_t$ is the graph over $\cC_{j+1}$ of $w$ with $\| w\|_{C^2( |x| \leq 3n)}^2 \leq C \, \delta_j^{\beta}$, it follows that  $|\nabla \bar{H}| \leq C \, \delta_j^{ \frac{\beta}{2}}$
and
$|\Pi_{j+1} (\partial_{x_i})| \leq C \, \delta_j^{ \frac{\beta}{2}}$ on $|x| \leq 3n$. Therefore,   \eqr{e:show1} follows from  
	\begin{align}	\label{e:show2a}
		\sup_{ |x| \leq {3n}} |\nabla_{\theta} \bar{H} | &\leq C \, \delta_j^{ \bar{\beta}} \, , 
	\end{align}
	where $\bar{H}$ is now regarded as a function on $\cC_{j+1}$ itself.
	It remains to establish   \eqr{e:show2a}.

	    The mean curvature $\bar{H}$ of the graph of $w$ is given at each point explicitly as a function of $w$, $\nabla w$ and $\Hess_w$; see
          corollary $A.30$ in \cite{CM2}.  We can write this as the first order part (in $w, \nabla w , \Hess_w$) plus a quadratic remainder
          \begin{align}	\label{e:expandH}
          	\bar{H}   &= H_{\cC} + \left( \Delta_{\theta} + \Delta_x +   \frac{1}{2} \right) w + O(w^2)     \, .
          \end{align}
           Here  $O(w^2)$ is a term that depends at least quadratically on $w, \nabla w , \Hess_w$ and the constant $H_{\cC} = \frac{ \sqrt{n-k} }{\sqrt{2}}$ is the mean curvature of $\cC$.  We will show that  the $\theta$ derivative of each of the terms in \eqr{e:expandH} is bounded by $C \, \delta_j^{ \bar{\beta}}$.   This is obvious for the constant term.  It is also obvious for the quadratic term $O(w^2)$ using interpolation and 
             $\| w\|_{C^2( |x| \leq {3n})}^2 \leq C \, \delta_j^{\beta}$.          Similarly, since $\beta \, \nu > \bar{\beta}$, the estimate \eqr{e:goodap2} and interpolation give that
          \begin{align}
          	\left|\nabla  \,  \left( \Delta_{\theta} + \Delta_x +   \frac{1}{2} \right) (w-\tilde{w}) \right| \leq C \, \delta_j^{ \bar{\beta}}  \, .
          \end{align}
          The proposition follows from the above since
 applying the linearized operator to $\tilde{w}$ gives
 \begin{align}
          	 \nabla_{\theta} \, \left( \Delta_{\theta} + \Delta_x +   \frac{1}{2} \right) \tilde{w} &= \nabla_{\theta} \sum_k \left( \Delta_{\theta} + \Delta_x + \frac{1}{2} \right)  x_k \,  h_k = 0 \, .
	          \end{align}        
       \end{proof}

\end{document}